\documentclass[a4paper,10pt]{amsart}
\usepackage{txfonts,dsfont,amsfonts}
\usepackage{mathrsfs}
\usepackage[arrow,matrix]{xy}
\usepackage{amsmath,amssymb,amscd,bbm,amsthm}

\numberwithin{equation}{section}

\theoremstyle{plain}
\newtheorem{theorem}{Theorem}[section]

\newtheorem{lemma}[theorem]{Lemma}
\newtheorem{proposition}[theorem]{Proposition}
\theoremstyle{definition}
\newtheorem{definition}[theorem]{Definition}
\newtheorem{example}[theorem]{Example}

\theoremstyle{remark}

\newtheorem{remark}[theorem]{Remark}

\newcommand{\N}{\mathbb{N}}
\newcommand{\Z}{\mathbb{Z}}

\DeclareMathOperator{\gldim}{\rm gldim}

\DeclareMathOperator{\lw}{\rm
LW}\DeclareMathOperator{\nw}{\rm
NW}\DeclareMathOperator{\lc}{\rm LC}
\DeclareMathOperator{\ext}{\underline{Ext}}

\DeclareMathOperator{\gkdim}{\rm GKdim }

\begin{document}

\title{Artin-Schelter regular algebras of dimension five \linebreak with two generators}

\author{G.-S. Zhou\; and \; D.-M. Lu\ }

\address{Department of Mathematics, Zhejiang University,
Hangzhou 310027, China}

\email{10906045@zju.edu.cn; \quad dmlu@zju.edu.cn }

\begin{abstract}
We study and classify Artin-Schelter regular algebras of dimension five with two generators under an additional $\mathbb Z^2$-grading by Hilbert driven Gr\"{o}bner basis computations. All the algebras we obtained are strongly
noetherian, Auslander regular, and Cohen-Macaulay. One of the results provides an answer to Fl{\o}ystad-Vatne's question
in the context of $\mathbb Z^2$-grading. Our results also achieve a connection between Lyndon words and Artin-Schelter regular algebras.
\end{abstract}




\maketitle





\section*{Introduction}

\newtheorem{maintheorem}{\bf{Theorem}}
\renewcommand{\themaintheorem}{\Alph{maintheorem}}
\newtheorem{maincorollary}[maintheorem]{\bf{Corollary}}
\renewcommand{\themaincorollary}{}
\newtheorem{mainquestion}[maintheorem]{\bf{Question}}
\renewcommand{\themainquestion}{}

The whole area of noncommutative projective geometry (in
the sense of Artin) began with the classification of Artin-Schelter regular
algebras of dimension 3 by Artin, Schelter, Tate and Van den Bergh
\cite{ASc, ATV1, ATV2} during 1987-1991. These algebras of dimension 3, then of dimension 4 have been extensively studied after that, and many interesting classes of Artin-Schelter regular algebras of
dimension 4 have been found.

To seek a substantial class of Artin-Schelter regular algebras, Palmieri, Wu, Zhang and the second author started the project to classify 4-dimensional Artin-Schelter regular algebras in \cite{LPWZ}. As the paper observes, Artin-Schelter regular algebras of dimension 4 which are domains have three resolution types according to the number of generators in degree 1. And they classified the type of these algebras with two generators by using $A_\infty$-algebraic methods under certain generic conditions. The type with three generators has been studied by Rogalski and Zhang in \cite{RZ},
where they gave all the families of Artin-Schelter regular algebras with an additional $\Z^2$-grading. The type with four generators is Koszul algebras, Zhang and Zhang introduced a new construction, which is called double Ore extension, and they found some
new families of this type (see \cite{ZZ1, ZZ2}).

The study of Artin-Schelter regular algebras of dimension 5 was initiated by Fl{\o}ystad and Vatne in 2011 (see \cite{FV}). In particular, they gave all the possible resolution types of these algebras with two generators of degree 1 according to the number of defining relations.
They showed that there are exactly three types of Hilbert series of such algebras when having three defining relations. Two of these three types can be realized by the enveloping algebras of 5-dimensional graded Lie algebras,
while the other one cannot be realized in such a way but they constructed an extremal algebra. Recently, Wang and Wu classified one of the type with three relations of degree 4 by using $A_\infty$-algebraic method under a generic condition (see \cite{WW}). It is open that whether there is an
Artin-Schelter regular algebra of dimension 5 with the resolution types for the number of defining relations being four or five. We refer to \cite[Theorem 5.6]{FV} for the details.

Though the whole project of classifying Artin-Schelter regular algebras is far from finished, the progress has been made on such algebras of dimensions 4 and 5. An interesting observation is that there is a common property in these newfound algebras (see \cite{FV, LPWZ, RZ, WW, ZZ1, ZZ2}) we mentioned above: all of them can be endowed with an appropriate
$\mathbb{Z}^2$-grading. This guides us in our exploration within the setting of $\mathbb{Z}^2$-grading.

In this paper we classify connected Artin-Schelter regular algebras of dimension 5 with two generators under an additional $\mathbb Z^2$-grading. The general ideas we used are in a way similar to the work of \cite{RZ}. Fundamental invariants of a connected graded algebra are its Hilbert series. A key step is to determine the obstructions of such algebras, which is achieved by using the method of Hilbert driven Gr\"{o}bner basis computations. The classification result of this paper is the following.
\begin{maintheorem}
\label{Thm A}{\it
Let $\/\mathfrak{X}$ be the family consisting of examples from $\mathcal{A}$ to $\mathcal{P}$ (see Sections 3 to 7). Then $\/\mathfrak{X}$ is, up to isomorphism and switching, a complete list of Artin-Schelter regular properly $\Z^2$-graded algebras of global dimension five which are domains generated by two elements and of Gelfand-Kirillov dimension at least four.}
\end{maintheorem}

The result continues to provide the evidence about the conjecture that Artin-Schelter regular algebras have certain
nice ring-theoretic and homological properties.

\begin{maintheorem}
\label{Thm B}{\it
Let $A$ be an Artin-Schelter regular properly $\Z^2$-graded algebra of global dimension five that is generated by two elements. If $A$ is a domain of $\gkdim A\geq4$, then it is strongly noetherian,
Auslander regular and Cohen-Macaulay.}
\end{maintheorem}

 The classification result also provides an answer to Fl{\o}ystad-Vatne's question.

\begin{maincorollary}
\label{Thm C}{\it
There exist Artin-Schelter regular algebras of global dimension five of the type $(4, 4, 4, 5, 5)$. There is no Artin-Schelter regular domain that being properly $\Z^2$-graded of the type $(4, 4, 4, 5)$.}
\end{maincorollary}

 This work has benefited from the approach of Lyndon words (see Definition \ref{definition-lyndon-word}). In fact, it is one of motivations for us to investigate Artin-Schelter regular algebras by Lyndon words. There are plenty of Artin-Schelter regular algebras of which the obstructions consisted of Lyndon words. In such cases, the Gelfand-Kirillov dimension and the global dimension for these algebras in the question will be same, say $d$, and hence $d-1\leq \#(G)\leq d(d-1)/2$.

\begin{mainquestion}
\label{question D}
{\it Let $k\langle x_1, x_2\rangle$ be a $\Z^2$-graded free algebra with $\deg(x_1)=(1, 0)$ and $\deg(x_2)=(0, 1)$, $\mathfrak{a}$ a homogeneous ideal of $k\langle x_1,x_2\rangle$ and $A=k\langle x_1,x_2\rangle/\mathfrak{a}$. Given an appropriate admissible order on $\{x_1,x_2\}^*$, does the obstructions of $A$ consist of Lyndon words if the algebra $A$ being an Artin-Schelter regular algebra?}
\end{mainquestion}

It turns out that the answer is positive whenever the algebra $A$ in the question is an Artin-Schelter regular algebra which is a domain either of global dimension $\leq4$ or of global dimension $5$ and GK-dimension $\geq4$. We expect to find new evidence to support the positive answer of the question.

The organization of the paper is as follows. In Section 1, we recall the
definition of Artin-Schelter regular algebras and review some basic facts concerning Gr\"{o}bner bases. In Section 2, we deal with the Hilbert series of connected $\Z^r$-graded algebras, the ideas used for the classification of Artin-Schelter regular algebras by the Hilbert driven Gr\"{o}bner basis computations are explained also in this section, those are essential to the paper. We devote next five sections to the classification of five types of the algebras. Our last section is to conclude the main theorems, and to discuss the question above.

Throughout the paper, $k$ is a fixed algebraically closed field of characteristic zero, all algebras are $k$-algebras generated in degree 1. The set of natural numbers $\N=\{ 0, 1, 2, \cdots \}$.

\section{Preliminaries}

In this section, we recall the definition of Artin-Schelter regular algebras, and the homogeneous Gr\"{o}bner basis theory as well as some of its elementary application in the setting of $\Z^s$-grading, where $\Z^s=\underbrace{\Z\times\cdots\times\Z}_s$ with the standard basis $\varepsilon_i=(0,\cdots,1,\cdots,0)$ for $i=1,2,\cdots,s$.

We use the following notations. For $\alpha=(\alpha_1,\cdots,\alpha_s)\in \Z^s$, the norm map $|\cdot|:\Z^s\to\Z$ is given by $|\alpha|=\alpha_1+\cdots+\alpha_s$. For arbitrary $\alpha,\beta\in \Z^s$, by $\alpha>\beta$ we mean $\alpha-\beta\in \N^s\backslash\{0\}$. For simplicity, we write $f(\textbf{t})$ for a multi-variable power series $f(t_1,\cdots,t_s)$, where $\textbf{t}^{\alpha}=t_1^{\alpha_1}\cdots t_s^{\alpha_s}$. Also denote $\omega:\Z^2\to\Z^2$ the switching map given by $\omega(m,n)=(n,m)$.

\subsection{\it Artin-Schelter regular $\Z^s$-graded algebras}\hspace{\fill}

A $\Z^s$-graded algebra $A=\bigoplus_{\alpha\in\Z^s}A_{\alpha}$ is called {\it connected} if $A_{\alpha}=0$ for $\alpha\not\in\N^s$ and $A_0=k$. A connected $\Z^s$-graded algebra $A$ is called \textit{properly} if $A$ is generated by $\bigoplus_{i=1}^s A_{\varepsilon_i}$ with $A_{\varepsilon_i}\neq0$ for all $i=1, 2,\cdots, s$. If $a\in A$ is a homogeneous element of degree $\alpha\in \Z^s$, then $|\alpha|\in \Z$ is called the {\it total degree} of $a$. The homological theory for connected $\Z^s$-graded algebras and modules is analogous to that of connected $\Z$-graded algebras and modules.

\begin{definition}
A connected $\Z^s$-graded algebra $A$ is called \textit{Artin-Schelter regular} (AS-regular, for short) of dimension $d$ if
\begin{enumerate}
\item[(AS1)]
$A$ has finite global dimension $d$;
\item[(AS2)]
$A$ has finite Gelfand-Kirillov dimension ($\gkdim$);
\item[(AS3)]$A$ is Gorenstein; that is, for some $\gamma\in\Z^s$,
\begin{eqnarray*}
\ext^i_A(k_A,A)=\left\{
\begin{array}{ll}
0,&i\neq d,\\
k(\gamma),& i=d,
\end{array}\right.
\end{eqnarray*}
where $k_A$ is the trivial right $A$-module $A/A_{>0}$, and the notation $(\gamma)$ is the degree $\gamma$-shifting on $\Z^s$-graded modules. The multi-index $\gamma$ will be called the \textit{Gorenstein parameter} of $A$.
\end{enumerate}
\end{definition}

For an arbitrary $\Z^s$-graded algebra $A$ and $\Z^s$-graded $A$-module $M$ we associate them with a $\Z$-graded algebra $A^{\rm gr}$ and a $\Z$-graded $A^{\rm gr}$-module $M^{\rm gr}$ defined by
$$
A^{\rm gr}=\bigoplus_{n\in \Z}\Big( \bigoplus_{|\alpha|= n} A_{\alpha} \Big),\qquad
M^{\rm gr}=\bigoplus_{n\in \Z}\Big( \bigoplus_{|\alpha|= n} M_{\alpha} \Big).
$$

If either $A$ or $A^{\rm gr}$ is AS-regular, then all modules $F_n$ in the minimal free resolution of $k_A$ are finitely generated (\cite[Proposition 3.1]{SZ}), which implies
$$
\underline{\rm Hom}_A(F_n,A)^{\rm gr} = \underline{\rm Hom}_{A^{\rm gr}}(F_n^{\rm gr},A^{\rm gr})
$$
for all $n\geq0$. This observation gives rise to the following result immediately.
\begin{lemma}
Let $A$ be a connected $\Z^s$-graded algebra. Then $A$ is AS-regular iff $A^{\rm gr}$ is AS-regular. Moreover, if $A$ is of Gorenstein parameter $\gamma$, then $A^{\rm gr}$ is of Gorenstein parameter $|\gamma|$.
\end{lemma}
In general, one can change a $\Z^s$-graded algebra $A$ into a $\Z^{\bar s}$-graded algebra $A^{\varphi}$ by an additive map $\varphi:\Z^s \to \Z^{\bar s}$ of abelian groups in a natural way. In particular, we call $A^\omega$ the {\it switching algebra\/} of $A$ when applying it to the switching map $\omega$ on the $\Z^2$-graded algebra.

\subsection{\it Gr\"{o}bner bases}\hspace{\fill}

In this subsection we firstly establish notations and briefly
review the noncommutative Gr\"{o}bner basis theory,
and then present some technical results that is used
in the sequel. A detailed treatment can
be found in \cite{Ge,Li,Mo}.

Throughout $X$ stands for a finite alphabet of letters. The free algebra on $X$ is denoted by $k\langle X\rangle$. The set of all words on $X$, denoted $X^*$ which contains the empty word ``1'', forms a basis of $k\langle X\rangle$. A word $u$ is nontrivial if $u\neq1$. For words $u$ and $v$, we say that $u$ is a \textit{factor} of $v$ if $w_1uw_2=v$ for some words $w_1,w_2\in X^*$. If $w_1=1$ then $u$ is called a \textit{prefix} of $v$, and if $w_2=1$ then $u$ is called a \textit{suffix} of $v$. A factor $u$ of $v$ is called \textit{proper} if $u\neq v$. A subset $V \subseteq X^*$ is called an {\it antichain} of words if $1\not\in V$ and every word $v\in V$ has no proper factor in $V$.

In the sequel we also fix an \textit{admissible ordering} $>$ on $X^*$, that is a well ordering on $X^*$ satisfying two properties: (i) $u>1$ for all nonempty word $u\in X^*$; (ii) $u>v$ implies $w_1uw_2>w_1vw_2$ for all words $u,v,w_1,w_2\in X^*$. For any nonzero polynomial $f\in k\langle X\rangle$, the \textit{leading word} $\lw(f)$ of $f$ is the largest word occurs in $f$, and the \textit{leading coefficient} $\lc(f)$ of $f$ is the coefficient of $\lw(f)$ in $f$. A nonzero polynomial with leading coefficient $1$ is said to be \textit{monic}. For any set $G\subseteq k\langle X\rangle$, denote
\begin{eqnarray*}
\lw(G)=\{\ u\in X^* : u =\lw(g) \text{ for some } g\in G\ \}, \\
\nw(G)=\{\ u\in X^* : u \; \text{ has no factor in } \lw(G)\; \}.\;
\end{eqnarray*}
A polynomial $f$ is called \textit{normal modulo $G$} if
every word occurred in $f$ belongs to $\nw(G)$. A polynomial $r$ is called a \textit{remainder of $f$ modulo $G$} if $r$ is normal modulo $G$ and
$$
f=\sum_{i\in I}a_iu_ig_iv_i+r,
$$
where each $a_i\in k$, $u_i,v_i\in X^*$, $g_i\in G\backslash\{0\}$
and $u_i\lw(g_i)v_i\leq\lw(f)$. Note that every polynomial has a remainder modulo $G$. If $0$ is a remainder of $f$ modulo $G$, then we say $f$ is \textit{trivial modulo $G$}.

\begin{definition}
Let $G$ be a set of polynomials in $k\langle X\rangle$.
If $\lw(f)\in\nw(G\backslash\{f\})$ for every
$f\in G$, then we say $G$ is \textit{tip-reduced}. If $f$
is monic and normal modulo $G\backslash\{f\}$ for every
$f\in G$, then we say $G$ is \textit{reduced}.
For an ideal $\mathfrak{a}$ of $k\langle X\rangle$, a
\textit{Gr\"{o}bner basis of $\mathfrak{a}$} is a subset $G\subseteq \mathfrak{a}$ such that $\lw(f)\not\in\nw(G)$ for every nonzero element $f\in\mathfrak{a}$.

Note that every ideal of $k\langle X\rangle$ admits a unique reduced Gr\"{o}bner basis, and it is generated by any one of its Gr\"{o}bner bases.
When $G$ is the reduced Gr\"{o}bner basis of $\mathfrak{a}$, $\lw(G)$ is called \textit{obstructions} of the algebra $k\langle X\rangle/\mathfrak{a}$.
\end{definition}

A $4$-tuple $(l_1,r_1,l_2,r_2)$ of words is called an \textit{ambiguity} of
a pair of words $(u_1,u_2)$ in case $l_1u_1r_1=l_2u_2r_2$ and
one of the following conditions holds: (1) $l_1=r_1=1$;
(2) $l_1=r_2=1$, $r_1$ is a nontrivial proper suffix of $u_2$
and $l_2$ is a nontrivial proper prefix of $u_1$.
We define a \textit{composition} of nonzero polynomials $f_1,f_2$ to be a polynomial of the form
\[S(f_1,f_2)[l_1,r_1,l_2,r_2]=\frac{l_1f_1r_1}{\lc(f_1)}-\frac{l_2f_2r_2}{\lc(f_2)},\] where $(l_1,r_1,l_2,r_2)$ is an ambiguity of $\big(\lw(f_1),\lw(f_2)\big)$.

The following result is a truncated version of Bergman's diamond lemma \cite{Berg}, and is crucial in the sequel. The result follows from a trivial variation of the discussion in the proof of \cite[Theorem 3.6.1]{Ge}. See also \cite[Theorem 5.1]{Mo} and \cite[Theorem 1.2]{Berg}.

\begin{lemma}
\label{truncated diamond lemma}
Assume that the free algebra $k\langle X\rangle$ is $\Z^s$-graded by a set map $e:X\to \N^s\backslash\{0\}$. Let $\mathfrak{a}$ be a homogeneous ideal of $k\langle X\rangle$ and $A=k\langle X\rangle/\mathfrak{a}$.
Let $G\subseteq \mathfrak{a}$ be any homogeneous set and
and $R=k\langle X\rangle/(\lw(G))$.
then for any $\alpha\in\N^s$ the following are equivalent:
\begin{enumerate}
\item $\lw(f)\not\in \nw(G)$ for any nonzero homogeneous element $f\in\mathfrak{a}$ with degree $\leq\alpha$.
\item Any homogeneous element $f\in\mathfrak{a}$ with degree $\leq\alpha$ is trivial modulo $G$.
\item The set of normal words of degree $\beta$ modulo $G$ is a basis of $A_{\beta}$ for every $\beta\leq\alpha$.
\item $\dim A_{\beta}=\dim R_\beta$ for every $\beta\leq\alpha$.
\item Any composition of nonzero polynomials in $G$ with degree $\leq\alpha$ is trivial modulo $G$ and $(G)_{\beta}=\mathfrak{a}_{\beta}$ for every $\beta\leq\alpha$.
\end{enumerate}
Moreover, $G$ is a Gr\"{o}bner basis of $\mathfrak{a}$ if and only if one of the equivalent conditions listed above holds for all $\alpha\in \N^s$.
\end{lemma}

\subsection{\it Invariants of $\Z^s$-graded algebras}
\hspace{\fill}

Throughout this subsection $k\langle X\rangle$ stands for a connected $Z^s$-graded algebra with grading given by a set map $e:X\to \N^s\backslash\{0\}$, $\mathfrak{a}$ is a homogeneous ideal of $k\langle X\rangle$ and $A=k\langle X\rangle/\mathfrak{a}$. We fix an admissible order on $X^*$ and let $G$ be the reduced Gr\"{o}bner basis of $\mathfrak{a}$. So $G$ consists of homogeneous elements. Anick constructed in \cite{An2} a resolution of the trivial module $k_A$, which gives rise to several efficient applications in both theory and calculations
of associative algebras, see \cite{An1,An2,GI1}. For reader's convenience, we recall some definitions and results from \cite{An1,An2,GI1} which are frequently used in the sequel, in particular in computing the Hilbert series and the global dimension of certain algebras.

The notion of $n$-chains, which is the ``bricks'' of Anick's resolution, was introduced by Anick in \cite{An2}. Later in \cite{Uf}, the author gives a slightly different but equivalent definition, and present an interpretation of chains in terms of graphs. To better understand and use Anick's resolution from a computational viewpoint, we quote the graphic interpretation as the definition of an $n$-chain.

\begin{definition}
Let $V$ be an antichain of words. The {\it graph of chains on $V$} is the directed graph $\Gamma(V)$, whose set of vertices consists of letters in $X \backslash V$ and all proper suffix of words in $V$ that is of length $\geq2$, and whose set of arrows is defined by saying  $u\to v$ if and only if $uv$ has a factor in $V$ and $uv'$ has no factor in $V$ for every proper prefix $v'$ of $v$. For $n\geq1$, a word $w$ is called an {\it $n$-chain on $V$} if there is an expression $w=v_1v_2\cdots v_n$ for some path $v_1\to v_2\to \cdots\to v_n$ in $\Gamma(V)$ with $v_1\in X$. A {\it $0$-chain on $V$} is the empty word $1$. Denote by $C_n(V)$ the set of all $n$-chains on $V$ for all $n\geq0$. Readily one has $C_{0}(V)=\{1\}$, $C_{1}(V)=X\backslash V$ and $C_{2}(V)=V\backslash X$. It is worth noting that the path factorization in $\Gamma(V)$ of an $n$-chain is uniquely determined.
\end{definition}

The $n$-chain defined here is named an $(n\!-\!1)$-chain in \cite{An1,An2,GI1}, and thus their results should be changed accordingly in our contexts. Let $C_n=C_n(\lw(G))$ for all $n\geq0$, then Anick's resolution of $k_A$, constructed in \cite[Theorem 1.4]{An2}, is a graded free resolution of the form
\begin{eqnarray}
\cdots
\xrightarrow{d_3} kC_2\otimes A
\xrightarrow{d_2} kC_1\otimes A
\xrightarrow{d_1} kC_0\otimes A
\to k_A
\to 0, \nonumber
\end{eqnarray}
which derives the ($s$-variable) Hilbert series and the global dimension of $A$ as
\begin{eqnarray}
\label{formula-hilbert-series}
H_A(\mathbf{t}) = \big(\textstyle\sum_{n\geq0} (-1)^n H_{kC_n}(\mathbf{t})\big)^{-1},  \nonumber\\
\gldim A  \leq \max \{\ n\geq0: C_n\neq\emptyset\ \}.\; \nonumber
\end{eqnarray}
In general, the Anick's resolution is not minimal, but it is minimal if $G$ is an antichain of words (see \cite[Lemma 3.3]{An2}), and hence in this situation one has
$$
\gldim A =\max \{\ n\geq0: C_n\neq\emptyset\ \}.
$$

The following result can be extracted from \cite[Theorem II]{GI1}, which provides an efficient way to calculate the global dimension of connected graded algebras.

\begin{theorem}
\label{theorem-global-dimension}
The notations and assumptions as above. Assume that $A$ has finite GK-dimension and $d=\max\{\ n\geq0:  C_n\neq\emptyset\ \} < \infty$, then the following hold:
\begin{enumerate}
\item $G$ is a finite set;
\item $\gldim A=\gkdim A=d$, and $A$ has polynomial growth of degree $d$;
\item $H_A(\textbf{t}) = \prod_{i=1}^d (1-\textbf{t}^{\alpha(i)})^{-1}$
    for some $\alpha(1),\cdots,\alpha(d) \in \N^s\backslash\{0\}$.
\end{enumerate}
\end{theorem}


Now we consider the situation when $\lw(G)$ has good combinatorial properties.

\begin{definition}
\label{definition-lyndon-word}
Assume $X=\{x_1,\cdots,x_n\}$. The \textit{lex order} on $X^*$ is given by saying $u>_{lex} v$ iff there are factorizations $u=rx_is,v=rx_jt$ with $i>j$. The \textit{plex order} on $X^*$ is given by saying that $u>_{plex}v$ iff either $u$ is a proper prefix of $v$ or $u>_{lex}v$. A word $u\in X^*$ is called a \textit{Lyndon word} if $u$ is nonempty and $u>_{lex} wv$ for every factorization $u=vw$ with $v,w\neq1$. The set of Lyndon words that is normal modulo $G$ is denoted by $\mathfrak{L}(G)$.
\end{definition}

\begin{remark}
The definition of plex order and Lyndon words follows the one given in \cite{BC}. In \cite{GF,Lo}, the \textit{plex order} on $X^*$ is defined by saying $u>_{plex}v$ iff either $u$ contains $v$ as a proper prefix or $u>_{lex}v$, and a word $u\in X^*$ is said to be Lyndon if $u$ is nonempty and $u<_{lex} wv$ for every factorization $u=vw$ with $v,w\neq1$. So the results in \cite{GF,Lo} need to be changed accordingly in our context. Mainly, one needs to replace $>_{plex}$ (resp. $\geq_{plex}$) by $<_{plex}$ (resp. $\leq_{plex}$). Overall, it is just a taste of flavor.
\end{remark}

A combination of  \cite[Theorem II]{GI1} and \cite[Theorem A,Theorem B]{GF} gives rise to the following result readily, which provides a more efficient way to calculate invariants of certain graded algebras.

\begin{theorem}
\label{theorem-Lyndon}
The notations and assumptions as above. Assume $X=\{x_1,\cdots,x_n\}$ and $\lw(G)$ consisting of Lyndon words, then the following hold:
\begin{enumerate}
\item
$\nw(G)=\{u_1u_2\cdots u_s\ |\ u_1\leq_{plex}u_2\leq_{plex}
\cdots\leq_{plex}u_s\in \mathfrak{L}(G),\ s\geq1\}\cup \{1\}$.
\item
$H_A(\textbf{t})=\prod_{u\in \mathfrak{L}(G)}(1-\textbf{t}^{\deg(u)})^{-1}$.
\item
$\gkdim A=\#(\mathfrak{L}(G))$. In particular, $A$ has polynomial growth iff
$\mathfrak{L}(G)$ is finite.
\item
If $G$ is a finite set, then $\gldim(A)\leq \#(G)+1$.
\item
If $A$ has polynomial growth, then $\gkdim A=\gldim A=\max\{\ n\geq0:  C_n\neq\emptyset\ \}<\infty$. Denote the common number by $d$,
then one has
$d-1\leq \#(G)\leq d(d-1)/2$.
\end{enumerate}
\end{theorem}

\subsection{\it Twisting algebra}
\hspace{\fill}

Let $A$ be a connected $\Z^s$-graded algebra and $\tau:A\to A$ a graded automorphism of algebras. We define a new $\Z^s$-graded algebra $A^{\tau}$ by $A^{\tau}=A$ as an $\Z^s$-graded $k$-module with a new multiplication defined by $a*b=a*_{\tau}b=a\tau^{|\alpha|}(b)$ for all $a\in A_{\alpha}$ and $b\in A_{\beta}$. We call $A^{\tau}$ the \textit{twisting $\Z^s$-graded algebra} of $A$ by $\tau$. It is easy to see that $(A^{\tau})^{\rm gr}=(A^{\rm gr})^\tau$.

Let $A,B$ be two $\Z^s$-graded algebras, $f:A\to B$ a graded homomorphism of $\Z^s$-graded algebras, let $\tau$ and $\sigma$ the $\Z^s$-graded automorphisms of $A$ and $B$, respectively. Then $f:A^{\tau} \to B^{\sigma}$ is a homomorphism of $\Z^s$-graded algebras whenever $f\circ\tau=\sigma\circ f$ holds. In particular, $\tau$ and $\tau^{-1}$ are also graded automorphisms of $A^{\tau}$ and $(A^{\tau})^{\tau^{-1}}=A$.
The following result is also crucial in the sequel. It helps us to reduce variables in the computations.

\begin{lemma}
\label{twisting}
Assume that the free algebra $k\langle X\rangle$ is $\Z^s$-graded with all $x\in X$ has total degree $1$. Let $G$ be a set of homogeneous polynomials and $A= k\langle X\rangle/(G)$. Let $\tau$ be a $\Z^s$-graded algebra automorphism of $k\langle X\rangle$ such that $\tau(G)\subseteq (G)$. Denote $\bar{\tau}$ the induced automorphism of $A$ by $\tau$, then the following hold:
\begin{enumerate}
\item $A^{\bar{\tau}}\cong k\langle X\rangle/(\phi_{\tau^{-1}}(G))$, where $\phi_{\tau^{-1}}: k\langle X\rangle \to k\langle X\rangle^{\tau^{-1}}$ is the homomorphism of algebras given by $\phi_{\tau^{-1}}(x_i)=x_i$ for all $x_i\in X$.
\item If further $\lw(\tau(x_i))=x_i$ for all $x_i\in X$ and $G$ is a Gr\"{o}bner basis of $(G)$, then $\phi_{\tau^{-1}}(G)$ is a Gr\"{o}bner basis of $(\phi_{\tau^{-1}}(G))$.
\item If further $\tau(x_i)$ is a scalar multiple of $x_i$ for all $x_i\in X$, then $\phi_{\tau^{-1}}(\Omega)$ is normal in $k\langle X\rangle/(\phi_{\tau^{-1}}(G))$ for any homogeneous polynomial $\Omega\in k\langle X\rangle$ that is normal in $A$.
\end{enumerate}
\end{lemma}

\begin{proof}
To see (1), we firstly note that $\phi_{\tau^{-1}}\circ \tau=\tau\circ \phi_{\tau^{-1}}$ because they are both algebra homomorphisms from $k\langle X\rangle$ to $k\langle X\rangle^{\tau^{-1}}$, and are equal on every letter $x_i\in X$.
So $\phi_{\tau^{-1}}:k\langle X \rangle^{\tau} \to (k\langle X \rangle^{\tau^{-1}})^{\tau}=k\langle X\rangle$ is also an algebra homomorphism. Let $\phi_{\tau}:k\langle X\rangle \to k\langle X\rangle^\tau$ be the algebra homomorphism given by $\phi_\tau(x_i)=x_i$ for all $x_i\in X$. Then we have
$
\phi_\tau\circ \phi_{\tau^{-1}}=\phi_{\tau^{-1}}\circ \phi_\tau=\text{id}_{k\langle X \rangle}
$
since all maps are algebraic endomorphisms of $k\langle X \rangle$ and are equal on each $x_i\in X$. A simple induction on total degrees of polynomials in $G$ will gives that there is an expression of the form, for every $g\in G$,
$$
\tau^{-1}(g) = \textstyle\sum_{i=1}^n f_i*_{\tau} g_i *_\tau h_i,
$$
where $g_i\in G$. Hence the ideal generated by $G$ in $k\langle X\rangle$ is equal to the ideal generated by $G$ in $k\langle X\rangle^\tau$ as graded vector spaces. It follows immediately that
$
k\langle X\rangle/(\phi_{\tau^{-1}}(G)) \xrightarrow{\cong} k\langle X\rangle^\tau/(G) = (k\langle X\rangle/(G))^{\bar{\tau}} = A^{\bar{\tau}},
$
where the first isomorphism is induced by $\phi_\tau$.

To see $(2)$, we assume that $f$ is a nonzero polynomial in $(\phi_{\tau^{-1}}(G))$. Then $\phi_{\tau}(f) \in (G)$ and hence
$
\lw(f)=\lw(\phi_\tau(f)) \not\in \nw(G)= \nw(\phi_{\tau^{-1}}(G)).
$
Now by definition $\phi_{\tau^{-1}}(G)$ is a Gr\"{o}bner basis of $(\phi_{\tau^{-1}}(G))$.

To see $(3)$, it suffices to show that $\Omega$ is normal in $A^{\bar{\tau}}$. By assumption one has $\tau(\Omega)=q\Omega$ for some nonzero $q\in k$. Suppose that $\Omega$ has total degree $m$, then any homogeneous element $a\in A$ of total degree $n$ gives rise to equations
$
a *_{\tau} \Omega = a \tau^n(\Omega) = q^n \Omega a' =\Omega *_\tau q^m \tau^{-m}(a')$ and $\Omega *_\tau a = \Omega \tau^m (a) = a'' \Omega = q^{-n}a''*_\tau \Omega
$
for some $a',a''\in A$. So $\Omega$ is also normal in $A^{\bar{\tau}}$.
\end{proof}

\section{Hilbert Driven Gr\"{o}bner basis computation}

In this paper, a key step is to determine the obstructions of algebras with respect to an appropriate admissible order. It is achieved by Hilbert driven Gr\"{o}bner basis computation. This section we deal with some notations and lemmas that are essential for our classification. We focus on the $\Z^2$-graded algebras, so homogeneous means $\Z^2$-homogeneous unless otherwise stated. Throughout $A$ is a properly $\Z^2$-graded algebra with two generators.

Let $k\langle x_1, x_2 \rangle$ stand for the $\Z^2$-graded free algebra with grading given by $\deg(x_1)=(1,0)$ and $\deg(x_2)=(0,1)$. For any nonzero homogeneous polynomial $f\in k\langle x_1,x_2\rangle_{(m,n)}$, we write $\deg f=(m, n)$, $\deg_1(f)=m$, $\deg_2(f)=n$, and call $|\deg(f)|=m+n$ the total degree of $f$.
We use the {\it deg-lex order $\/ >_{deglex}$ on $\Z^2$}. For arbitrary $\alpha=(\alpha_1, \alpha_2),\ \beta=(\beta_1, \beta_2)\in \Z^2$, by $\alpha >_{deglex} \beta$ we mean either $|\alpha|>|\beta|$ or $|\alpha|=|\beta|$ with $\alpha_2>\beta_2$. The admissible order on $\{x_1, x_2\}^*$ is chosen to be the {\it deg-lex order} given by: for $u, v\in \{x_1, x_2\}^*$,
\begin{eqnarray*}
\label{definition-deglex}
u >_{deglex} v\quad \text{ iff }\quad\left\{
\begin{array}{llll}
\deg(u)>_{deglex}\deg(v),\quad \text{or}&&\\
\deg(u)=\deg(v)\; \; \text{and}\;\; u>_{lex}v.
\end{array}\right.
\end{eqnarray*}

We choose a presentation $A=k\langle x_1,x_2\rangle/\mathfrak{a}$, where $\mathfrak{a}$ is a homogeneous ideal of $k\langle x_1, x_2 \rangle$.
Denote $T: k\langle x_1, x_2\rangle \to k\langle x_1, x_2\rangle$ the homomorphism of algebras given by $T(x_1)=x_2,\, T(x_2)=x_1$. Then the switching algebra $A^{\omega}$ has a presentation $A^{\omega}=k\langle x_1, x_2 \rangle/T(\mathfrak{a})$. Let $G$ be the unique reduced Gr\"{o}bner basis of $\mathfrak{a}$ and $R=k\langle x_1,x_2\rangle/(\lw(G))$. Then $G$ consists of homogeneous polynomials. Since the leading word of elements in $G$ are mutually distinct, there are at most a finite number of relations in $G$ of any degree. Thus elements in $G$ can be enumerated as
$$
f_1,\ f_2,\ f_3,\cdots \quad \text{by} \quad \lw(f_1) <_{deglex} \lw(f_2) <_{deglex} \lw(f_3) <_{deglex} \cdots.
$$
Furthermore, we may always assume that
$$
\deg_1(f_1)\geq\deg_2(f_1),
$$
otherwise, we can pass to the switching algebra $A^{\omega}$.

For each $i\geq0$, denote
$$
G^{i}=\{f_1, f_2, \cdots,f_i\},\quad A^i=k\langle x_1,x_2\rangle/(G^i),
\quad R^i=k\langle x_1,x_2\rangle/(\lw(G^i))
$$
where we have set $G^0=\emptyset$, the empty set, so $R^0=A^0=k\langle x_1,x_2\rangle$. Also denote
$$G_{min}=\{f_i\in G: f_i\not\in (G^{i-1})\}.$$

The following two lemmas are considered as bases for our deduction in determining $\lw(G)$. For two power series, by $f(t_1, t_2)\ge g(t_1, t_2)$ we mean all coefficients of $f(t_1, t_2)- g(t_1, t_2)$ are nonnegative.
\begin{lemma}
\label{lemma-hilbert-compare}
The following statements hold:
\begin{enumerate}
\item $H_R(t_1,t_2)=H_A(t_1,t_2)$ and $H_{R^i}(t_1,t_2)\geq H_{A^i}(t_1,t_2)$.
\item $H_{R^0}(t_1,t_2)\geq H_{R^1}(t_1,t_2)\geq\cdots\geq H_R(t_1,t_2)$.
\item $G=G^i$ iff $H_{R^i}(t_1,t_2)=H_R(t_1,t_2)$.
\item $G_{min}$ is a minimal generating set of the ideal $\mathfrak{a}$.
\end{enumerate}
\end{lemma}

\begin{proof}
(1), (2) and (3) are trivial. To prove (4), we need to show that $G_{min}$ generates the ideal $(G)$ and every proper subset of $G_{min}$ does not generate $(G)$.

Suppose that $G\not\subseteq (G_{min})$, and let $i$ be the smallest integer with  $f_i\notin (G_{min})$.
Now $G^{i-1}\subseteq (G_{min})$ implies that $f_i\notin (G^{i-1})$, this causes a contradiction of $f_i\in G_{min}$ by the definition. So $G_{min}$ generates the ideal $(G)$.
If $G_{min}$ is not minimal, then there exists some $\alpha\in\N^2$ and a largest integer $n$ such that
$\deg(f_n)=\alpha$ and $f_n\in (G_{min}\backslash\{f_n\})$. Then one has $f_n\in (G^{n-1})$ and hence
$f_n\not\in G_{min}$, which is impossible.
So $G_{min}$ is a minimal generating set of the ideal $(G)$.
\end{proof}

\begin{lemma}
\label{lemma-new-relation}
Let $\beta =(\beta_1,\beta_2)\in \N^2$ be a fixed index. If $H_{R^i}(t_1,t_2) - H_{R}(t_1,t_2) = p t_1^{\beta_1}t_2^{\beta_2} + \sum_{\alpha\not\leq \beta} r_\alpha t_1^{\alpha_1}t_2^{\alpha_2}$ with $p>0$. Then $G\backslash G^i$ contains exactly $p$ distinct relations of degree $\beta$. Similarly the single variable version for $\Z$-grade algebra is also true.
\end{lemma}

\begin{proof}
Note that $\dim R_\beta$ equals the number of normal words of degree $\beta$ modulo $G$, and $\dim R_\beta^i$ equals the number of normal words of degree $\beta$ modulo $G^i$. So there are exactly $p$ distinct words, say $\{w_1, w_2,\cdots,w_p\}$, of degree $\beta$ that are normal modulo $G^i$ but not normal modulo $G$. So $G$ contains $f_{i_1}, f_{i_2}, \cdots,f_{i_p}$ with all $i_s>i$ such that $\lw(f_{i_s})$ is a factor of $w_s$ for $s=1, 2, \cdots,p$. However, the difference $H_{R^i}(t_1,t_2) - H_{R}(t_1,t_2)$ also tells that $G\backslash G^i$ contains no relations of degree $<\beta$. Thus $\lw(f_{i_s})=w_s$ and hence $\deg(f_{i_s}) = \beta$ for $s=1, 2, \cdots,p$. This completes the proof.
\end{proof}

\begin{remark}
The basic idea of ``Hilbert driven Gr\"{o}bner basis computation'' is to guide the computation of $\lw(G)$ by a priori knowledge of Hilbert series $H_A(t_1,t_2)$. Roughly speaking, the algorithm in our application is performed as follows:
begin with $\lw(G^0)=\emptyset$; drive $\lw(G^i)$ step-by-step
by examining the Hilbert series $H_{R^i}(t_1,t_2) - H_{R}(t_1,t_2)$ according to Lemma \ref{truncated diamond lemma}, Lemma \ref{lemma-hilbert-compare}, Lemma \ref{lemma-new-relation} and various priori given properties on $A$; terminate the procedure till $H_{R^i}(t_1,t_2) - H_{R}(t_1,t_2)$ vanished.
\end{remark}

Now we focus on the main subject. In the remaining of this paper, we always assume that $A$ is an AS-regular algebra of global dimension 5 which is a domain of GK-dimension $\geq 4$ unless otherwise stated.
So no element of $G$ is of degree $(m,0)$ or $(0,n)$ due to the assumption.

The next lemma is elementary in our analyzing the ($\Z^2$-grading) resolution types of $k_A$, which is used to calculate the Hilbert series $H_A(t_1,t_2)$.

\begin{lemma}
\label{minimal resolution}
Suppose that the minimal free resolution of $k_A$ is of the form
\begin{eqnarray*}
0 \to
\begin{array}[t]{c}
A\\[-0.5em] \text{\tiny\rm $(p, q)$}
\end{array}
\xrightarrow{d_5}
\begin{array}[t]{c}
A^2\\[-0.5em] \text{\tiny \rm $(p, q\!-\!\!1)$}\\[-0.7em] \text{\tiny \rm $ (p\!-\!\!1, q)$}
\end{array}
\xrightarrow{d_4}
\begin{array}[t]{c}
A^m\\[-0.5em] \text{\tiny \rm $(p\!-\!u_1, q\!-\!\!v_1)$}\\[-0.7em]\text{\tiny \rm $(p\!-\!u_2, q\!-\!\!v_2)$}\\[-0.7em]
\text{\tiny \rm $\vdots$}\\[-0.7em]\text{\tiny \rm $(p\!-\!u_m, q\!-\!\!v_m)$ }
\end{array}
\xrightarrow{d_3}
\begin{array}[t]{c}
A^m\\[-0.5em] \text{\tiny \rm $(u_1, v_1)$}\\[-0.7em] \text{\tiny \rm $(u_2, v_2)$}
\\[-0.7em] \text{\tiny \rm $\vdots$ }
\\[-0.7em] \text{\tiny \rm $(u_m, v_m)$}
\end{array}
\xrightarrow{d_2}
\begin{array}[t]{c}
A^2\\[-0.5em] \text{\tiny \rm (1, 0)}\\[-0.7em] \text{\tiny \rm (0, 1)}
\end{array}
\xrightarrow{d_1}
\begin{array}[t]{c}
A\\[-0.5em] \text{\tiny \rm (0, 0)}
\end{array}
\to k_A
\to 0,
\end{eqnarray*}
where $m\geq1$ and $p, q, u_i, v_i\in\Z$. Let $\sigma$ and $\tau$ be two permutations on $\{1, 2, \cdots, m\}$ such that $u_{\sigma (1)}\leq u_{\sigma(2)}\leq\cdots\leq u_{\sigma(m)}$ and $v_{\tau(1)}\leq v_{\tau(2)}\leq \cdots\leq v_{\tau(m)}$, respectively. Then
\begin{enumerate}
\item $u_{\sigma(1)}\ge 1, \; u_{\sigma(1)}+u_{\sigma(m)}<p$, and $u_{\sigma(2)}+u_{\sigma(m)}\leq p$.
\item $v_{\tau(1)}\ \ge 1,\; v_{\tau(1)}\ +v_{\tau(m)}\ <q$, and $v_{\tau(2)}+v_{\tau(m)}\ \leq q$.
\end{enumerate}
\end{lemma}

\begin{proof}
We only justify (1). The proof of (2) is similar. Write the differentials $d_1, d_2,\cdots, d_5$ in matrix form. So the first three of them are
\begin{eqnarray*}
d_1=(x_1,x_2),\quad
d_2=\left(
\begin{array}{cccc}
a_{11}&a_{12}&\cdots&a_{1m}\\
a_{21}&a_{22}&\cdots&a_{2m}
\end{array}\right), \quad
d_3=\left(
\begin{array}{cccc}
b_{11}&b_{12}&\cdots&b_{1m}\\
b_{21}&b_{22}&\cdots&b_{2m}\\
\cdots&\cdots&\cdots&\cdots\\
b_{m1}&b_{m2}&\cdots&b_{mm}
\end{array}\right)
\end{eqnarray*}
where each nonzero entry in $d_2,\, d_3$ is homogeneous of total degree $\geq1$.
 Since $A$ is a domain and
the resolution is minimal, each column of $d_2,\; d_3$ has at least two
nonzero entries of total degree $\geq1$. Thus one gets $
u_{\sigma(1)}\ge 1$, and  both $u_{\sigma(1)}$ and $u_{\sigma(2)}\leq p-u_{\sigma(m)}$. We claim that $u_{\sigma(1)}\ne p-u_{\sigma(m)}$. Otherwise, the $\sigma(m)$-th column of $d_3$ is of the form
$$
(\lambda_1x_2^{s_1}, \lambda_2x_2^{s_2}, \cdots, \lambda_m x_2^{s_m})^T
$$
with $s_1, s_2, \cdots, s_m\geq1$ and $\lambda_1, \lambda_2, \cdots, \lambda_m\in k$.
By change of rows, one can assume $s_1\leq s_2\leq\cdots\leq s_m$ and
$\lambda_1, \lambda_2\neq0$.
Let
$$
P=I_m-\sum_{w=2}^mE_{w1} (\lambda_1^{-1}\lambda_2x_2^{s_w-s_1}),
$$
where $I_m$ is the $m\times m$-identity matrix and $E_{ij}$ is the $m\times m$-matrix
with entry $1$ at position $(i,j)$ and $0$ at others.
Since $P$ is invertible with
$P^{-1}=I_m+\sum_{w=2}^mE_{w1} (\lambda_1^{-1}\lambda_2x_2^{s_w-s_1})$,
after replacing $d_2$ by $d_2P^{-1}$ and $d_3$ by $Pd_3$, respectively, we still get a
minimal resolution of $k$. However, the $\sigma(m)$-th column of the new $d_3$ is of the form
$$
(\lambda_1x_2^{s_1}, 0, \cdots, 0)^T
$$
and hence $a_{11}(\lambda_1x_2^{s_1})=0$, which contradicts the assumption of that $A$ is a domain.
\end{proof}

According to \cite{FV}, there are five possible types of such AS-regular algebras sorted by the total degree distribution of the relations in $G_{min}$, named $(3, 5, 5),\ (3, 4, 7),\ (4, 4, 4),\ (4, 4, 4, 5)$ and $(4, 4, 4, 5, 5)$. They showed there are exactly such algebras of first three types. The existence of the types $(3, 5, 5)$ and $(4, 4, 4)$ can be realized by the enveloping algebras of 5-dimensional graded Lie algebras. Later, the type $(4, 4, 4)$ was classified in \cite{WW} by using $A_\infty$-algebraic method under a generic condition. Further, Fl{\o}ystad and Vatne constructed an interesting example to show the existence of the type $(3, 4, 7)$, and they asked a question for seeking the algebras of the types $(4, 4, 4, 5)$ and $(4, 4, 4, 5, 5)$.

In the following sections, we will encounter several families of algebras and state some of basic facts without proof to avoid duplication. We refer the readers to the following hints:
\begin{enumerate}
\item A claim about twist equivalence follows from Lemma \ref{twisting}.
\item A claim about a homogeneous polynomial being normal in $A$ is checked by a tedious but straightforward reductions modulo $G$. Similarly for normal sequences in $A$.
\item A claim about a homogeneous polynomial,
say $z$ of total degree $n$, being regular in $A$ follows from the equation $H_{A^{\rm gr}/(z)}(t)= (1-t^{n})H_{A^{\rm gr}}(t)$.
\item A claim about global dimension of $A$ follows from
Theorem \ref{theorem-global-dimension} or Theorem \ref{theorem-Lyndon}.
\item A claim about AS-regularity, Auslander-regularity, Cohen-Macaulayness and strongly notherian of $A$ follows from \cite[Lemma 1.3]{RZ}.
\end{enumerate}

In the classification of each type, we use some Maple programs
\footnote{Maple results are omitted due to its length, the reader wishing to verify our computations can find the Maple code for these programs,
which we make freely available, on the website:

http://www.math.zju.edu.cn/teacher\_intro.asp?userid=17\&item=\%D6\%F7\%D2\%AA\%C2\%DB\%D6\%F8 } to help us find remainders of polynomials in the free algebra modulo given relations and use Maple Solving Command to solve various equation systems. We have justified that all equation systems, except for the one occurred in the classification of type (4, 4, 4), can be solved without Maple Solving Command and give the same set of solutions. We also write programs in Maple to find some normal elements or normal sequence of low degree, which are essential to prove that all algebras we obtained have nice ring-theoretic and homological properties.

\section{Classification of type $(3, 5, 5)$}

We follow the notation and convention in the previous sections. This section we begin to classify 5-dimensional AS-regular $\Z^2$-graded algebras of type $(3, 5, 5)$, these are determined by three defining relations of the (total) degrees 3, 5, 5, respectively. We obtain 5 classes of AS-regular algebras listed by $\mathcal{A}, \mathcal{B}, \mathcal{C}, \mathcal{D}$ and $\mathcal{E}$.

Thanks to \cite{FV}, the minimal free resolution of the trivial module $k$ over $A^{\rm gr}$ is of the form
\begin{eqnarray*}
0
\to A^{\rm gr}(-11)
\to A^{\rm gr}(-10)^2
\to A^{\rm gr}(-6)^2  \oplus A^{\rm gr}(-8)\hskip39mm \\
\to
A^{\rm gr}(-3)\oplus A^{\rm gr}(-5)^2
\to A^{\rm gr}(-1)^2
\to A^{\rm gr}
\to k_A
\to 0,
\end{eqnarray*}
and the Hilbert series of $R^{\rm gr}$ (the same one of $A^{\rm gr}$) is ${\big((1-t)^2(1-t^2)(1-t^3)(1-t^4)\big)}^{-1}$. Thus
$$
H_{(R^0)^{\rm gr}}(t)-H_{R^{\rm gr}}(t)=t^3+\sum_{n\geq4}r_nt^n
$$
and so $G$ has exactly one relation of total degree $3$,
which is $f_1$.  Since $\deg_1(f_1)\geq\deg_2(f_1)$, one has
\begin{eqnarray}
\deg(f_1)=(2,1) \; \text{ with }\; \lw(f_1)=x_2x_1^2, \nonumber
\end{eqnarray}
which gives rise to
$$
H_{(R^1)^{\rm gr}}(t)-H_{R^{\rm gr}}(t)=2t^5+\sum_{n\geq6}r_nt^n.
$$
Then $G$ has exactly two elements of total degree $5$, which are
$f_2, f_3$, and $G_{min}=\{f_1, f_2, f_3\}.$
By checking normal words of degree
$(4, 1), (3, 2), (2, 3), (1, 4)$ modulo $G^1$, respectively,
one gets
$$
\deg(f_2)=(2, 3)\ \text{ and }\ \deg(f_3)\in\{(2, 3),\ (1, 4)\}.
$$

\begin{lemma}
$\deg(f_3)=(1, 4)$ with $\lw(f_3)=x_2^4x_1$.
\end{lemma}

\begin{proof}
If $\deg(f_3)=(2, 3)$, then by Lemma \ref{minimal resolution},
the only possible minimal free resolution of $k_A$ is of the form
\begin{equation*}
0 \to
\begin{array}[t]{c}
A\\[-0.5em] \text{\tiny\rm (5, 6)}
\end{array}
\to
\begin{array}[t]{c}
A^2\\[-0.5em] \text{\tiny \rm (5, 5)}\\[-0.7em] \text{\tiny \rm (4, 6)}
\end{array}
\to
\begin{array}[t]{c}
A^3\\[-0.5em] \text{\tiny \rm (3, 3)}\\[-0.7em] \text{\tiny \rm (3, 3)}\\[-0.7em] \text{\tiny \rm (3, 5)}
\end{array}
\to
\begin{array}[t]{c}
A^3\\[-0.5em] \text{\tiny \rm (2, 1)}\\[-0.7em] \text{\tiny \rm (2, 3)}\\[-0.7em] \text{\tiny \rm (2, 3)}
\end{array}
\to
\begin{array}[t]{c}
A^2\\[-0.5em] \text{\tiny \rm (1, 0)}\\[-0.7em] \text{\tiny \rm (0, 1)}
\end{array}
\to
\begin{array}[t]{c}
A\\[-0.5em] \text{\tiny \rm (0, 0)}
\end{array}
\to k_A
\to 0.
\end{equation*}
By checking all normal words of degree $(2, 3)$ modulo $G^1$, there are three possibilities:

(1) $\lw(f_2)=x_2x_1x_2x_1x_2,\ \lw(f_3)=x_2x_1x_2^2x_1$. Then
$$
H_{R^3}(t_1,t_2)-H_{R}(t_1,t_2)=-t_1^3t_2^3+
\sum_{|\alpha|\geq7}r_{\alpha}t_1^{\alpha_1}t_2^{\alpha_2}.
$$

(2) $\lw(f_2)=x_2x_1x_2x_1x_2,\ \lw(f_3)=x_2^2x_1x_2x_1$. Then
$$
H_{R^3}(t_1,t_2)-H_{R}(t_1,t_2)=-t_1^3t_2^3+t_1^2t_2^4+
\sum_{|\alpha|\geq7}r_{\alpha}t_1^{\alpha_1}t_2^{\alpha_2}.
$$

(3) $\lw(f_2)=x_2x_1x_2^2x_1,\ \lw(f_3)=x_2^2x_1x_2x_1$. Then
$$
H_{R^3}(t_1,t_2)-H_{R}(t_1,t_2)=t_1^3t_2^4+
\sum_{|\alpha|\geq8}r_{\alpha}t_1^{\alpha_1}t_2^{\alpha_2}.
$$
So $G$ has a unique element of total degree $7$,
which is $f_4$ of degree $(3, 4)$.
By checking all normal words of degree $(3, 4)$ modulo $G^3$,
one has $\lw(f_4)=x_2x_1x_2x_1x_2x_1x_2$ and $f_1, f_2, f_3$ are of the following form:
\begin{eqnarray*}
\begin{split}
&f_1=x_2x_1^2+ a_1x_1x_2x_1+a_2x_1^2x_2,\\
&f_2=x_2x_1x_2^2x_1+b_2x_2x_1x_2x_1x_2+b_3x_1x_2^3x_1
+b_4x_1x_2^2x_1x_2+b_5x_1x_2x_1x_2^2+b_6x_1^2x_2^3,\\
&f_3=x_2^2x_1x_2x_1+c_2x_2x_1x_2x_1x_2+c_3x_1x_2^3x_1
+c_4x_1x_2^2x_1x_2+c_5x_1x_2x_1x_2^2+c_6x_1^2x_2^3.\\
\end{split}
\end{eqnarray*}
Apply Lemma \ref{truncated diamond lemma}, the compositions
$ S(f_2, f_1)[1, x_1, x_2x_1x_2, 1]$ and $S(f_3, f_1)[1, x_1, x_2^2x_1, 1]$
should be trivial modulo $G^3$, so $b_2=a_1$ and $c_2=-a_1^2$.
However, the remainder of the composition
$S(f_2, f_3)[1, x_2x_1, x_2x_1, 1]$ should be a nonzero scalar multiple of $f_4$,
which implies that $c_2\neq-b_2^2$.

All these three possibilities give contradictions and so we finish the proof.
\end{proof}
Now the following proposition is clear from Lemma \ref{minimal resolution}.
\begin{proposition}
\label{resolution-type-(3,5,5)}
The minimal free resolution of $k_A$ is of the form:
\begin{equation*}
0 \to
\begin{array}[t]{c}
A\\[-0.5em] \text{\tiny\rm (4, 7)}
\end{array}
\to
\begin{array}[t]{c}
A^2\\[-0.5em] \text{\tiny \rm (4, 6)}\\[-0.7em] \text{\tiny \rm (3, 7)}
\end{array}
\to
\begin{array}[t]{c}
A^3\\[-0.5em] \text{\tiny \rm (3, 3)}\\[-0.7em] \text{\tiny \rm (2, 4)}\\[-0.7em] \text{\tiny \rm (2, 6)}
\end{array}
\to
\begin{array}[t]{c}
A^3\\[-0.5em] \text{\tiny \rm (2, 1)}\\[-0.7em] \text{\tiny \rm (2, 3)}\\[-0.7em] \text{\tiny \rm (1, 4)}
\end{array}
\to
\begin{array}[t]{c}
A^2\\[-0.5em] \text{\tiny \rm (1, 0)}\\[-0.7em] \text{\tiny \rm (0, 1)}
\end{array}
\to
\begin{array}[t]{c}
A\\[-0.5em] \text{\tiny \rm (0, 0)}
\end{array}
\to k_A
\to 0.
\end{equation*}
\end{proposition}

\begin{lemma}
$\lw(f_2)=x_2^2x_1x_2x_1$.
\end{lemma}

\begin{proof}
By checking all normal words of degree $(2, 3)$ modulo $G^1$, we see that
$\lw(f_2)$ belongs to the set of words $\{x_2x_1x_2x_1x_2,\ x_2x_1x_2^2x_1,\ x_2^2x_1x_2x_1\}$.
It is impossible that $\lw(f_2)=x_2x_1x_2x_1x_2$, because otherwise it gives a contradiction:
$$
H_{R^3}(t_1, t_2)-H_{R}(t_1, t_2)=-t_1^3t_2^3+
\sum_{|\alpha|\geq7}r_{\alpha}t_1^{\alpha_1}t_2^{\alpha_2}.
$$
It remains to exclude that $\lw(f_2)=x_2x_1x_2^2x_1$. If so, then
$$
H_{R^3}(t_1, t_2)-H_{R}(t_1, t_2)=t_1^3t_2^5+
\sum_{|\alpha|\geq9}r_{\alpha}t_1^{\alpha_1}t_2^{\alpha_2}.
$$
Thus $G$ has a unique element of total degree $8$, which is $f_4$ of
degree $(3, 5)$. Apply Lemma \ref{truncated diamond lemma},
the remainder of the composition
$S(f_2, f_2)[1, x_2^2x_1, x_2x_1x_2, 1]$
modulo $G^3$ should be a nonzero scalar multiple of $f_4$.
By checking normal words of degree $(3,5)$ modulo $G^3$, we have
$\lw(f_4)\in \{x_2x_1x_2x_1x_2^3x_1, x_2x_1x_2x_1x_2x_1x_2^2\}$,
and both choices of $\lw(f_4)$ gives the contradiction
$$
H_{R^4}(t_1, t_2)-H_{R}(t_1, t_2)=-t_1^3t_2^6+
\sum_{|\alpha|\geq10}r_{\alpha}t_1^{\alpha_1}t_2^{\alpha_2}.
$$
Now we reach the only possibility $\lw(f_2)=x_2^2x_1x_2x_1$ as required.
\end{proof}

\begin{proposition}
\label{type-(3,5,5)}
$G=\{f_1, f_2, f_3, f_4\}$ and $G_{min}=\{f_1, f_2, f_3\}$ with
{\small\begin{eqnarray*}
\begin{split}
&f_1=x_2x_1^2+ a_1x_1x_2x_1+a_2x_1^2x_2,\\
&f_2=x_2^2x_1x_2x_1+b_1x_2x_1x_2^2x_1+b_2x_2x_1x_2x_1x_2+b_3x_1x_2^3x_1
+b_4x_1x_2^2x_1x_2+b_5x_1x_2x_1x_2^2+b_6x_1^2x_2^3,\\
&f_3=x_2^4x_1+c_1x_2^3x_1x_2+c_2x_2^2x_1x_2^2+c_3x_2x_1x_2^3+c_4x_1x_2^4,\\
&f_4=x_2^3x_1x_2^2x_1+d_1x_2^2x_1x_2^3x_1+d_2x_2^2x_1x_2^2x_1x_2 +d_3x_2x_1x_2^3x_1x_2 + d_4x_2x_1x_2^2x_1x_2^2+d_5x_2x_1x_2x_1x_2^3 \\&\hskip10mm
+d_6x_1x_2^3x_1x_2^2 +d_7x_1x_2^2x_1x_2^3+d_8x_1x_2x_1x_2^4+d_9x_1^2x_2^5,
\end{split}
\end{eqnarray*}}
where $a_i, b_i, c_i, d_i\in k$. Moreover one has $a_2c_4\neq0$ and $b_1\neq c_1$.
\end{proposition}

\begin{proof}
By previous discussions, we have $G_{min}=\{f_1, f_2, f_3\}$ and $f_1, f_2, f_3$ are of the given form by checking normal words of corresponding degrees modulo $G^3$ (which are same modulo $G$). Since $A$ is a domain, we have $a_2c_4\neq0$. It remains to show that $b_1\neq c_1$ and $G=\{f_1, f_2, f_3, f_4\}$ with $f_4$ of the giving form. Since
$$
H_{R^3}(t_1, t_2)-H_{R}(t_1, t_2)=t_1^2t_2^5+
\sum_{|\alpha|\geq8}r_{\alpha}t_1^{\alpha_1}t_2^{\alpha_2}.
$$
So $G$ has exactly one relation of total degree $7$,
which is $f_4$ of degree $(2, 5)$. By checking all normal words
of degree $(2, 5)$ modulo $G^3$, we have $\lw(f_4)\in \{\, x_2^2x_1x_2^3x_1,\,  x_2^3x_1x_2^2x_1\,\}$. It is impossible that $\lw(f_4)=x_2^2x_1x_2^3x_1$, for otherwise it gives a contradiction
$$
H_{R^4}(t_1, t_2)-H_{R}(t_1, t_2)=-t_1^2t_2^6+
\sum_{|\alpha|\geq9}r_{\alpha}t_1^{\alpha_1}t_2^{\alpha_2}.
$$
Thus we have $\lw(f_4)=x_2^3x_1x_2^2x_1$, which follows that $H_{R^4}(t_1, t_2)=H_{R}(t_1, t_2)$ and hence $G=G^4$. Now by checking normal words of degree $(2, 5)$ modulo $G$, one gets that $f_4$ is of the given form. Applying Lemma \ref{truncated diamond lemma}, the remainder of the composition $S(f_3, f_2)[1, x_2x_1, x_2^2, 1]$ modulo $G^3$ is a non-zero scalar multiple of $f_4$ in $k\langle x_1, x_2\rangle$. So, $b_1\neq c_1$ follows.
\end{proof}

Now let's turn to find possible solutions of the coefficients of $f_i$ in Proposition \ref{type-(3,5,5)}. Firstly note that there are six compositions of polynomials in $G$, namely
{\small\begin{eqnarray*}
\begin{array}{llll}
&S(f_2, f_1)[1, x_1, x_2^2x_1, 1],&S(f_3, f_1)[1, x_1, x_2^3, 1],
&S(f_3, f_2)[1, x_2x_1, x_2^2, 1],
\\
&S(f_3, f_4)[1, x_2^2x_1, x_2, 1], &S(f_4, f_1)[1, x_1, x_2^3x_1x_2, 1], &S(f_4, f_2)[1, x_2x_1, x_2^3x_1, 1].
\end{array}
\end{eqnarray*}}
We use a simple Program in Maple to help us calculate the remainder of these compositions modulo $G$ in $k\langle x_1, x_2\rangle$, and denote the results by $r_1, r_2, \cdots, r_6$, respectively.

Apply Lemma \ref{truncated diamond lemma}, as a necessary condition for $G$ to be Gr\"{o}bner, all the coefficients of $r_i$'s should be zero. By Lemma \ref{twisting}, it is equivalent to break the system of equations given by setting the coefficients of $r_i$'s to $0$, together with $a_2c_4\neq0,\ b_1\neq c_1$ into three situations:
\begin{enumerate}
\item[Case 1:] $a_1=0,\, b_1=0,\ c_3=1$. Maple Solving Command gives no solutions.
\item[Case 2:] $a_1=0,\, b_1=1$. Maple Solving Command gives $2$ families of solutions. Twist the $b_1$ back and include them as a parameter $p$, we get the families of solutions $\mathcal{A},\, \mathcal{B}$ below.
\item[Case 3:] $a_1=1$.  Maple Solving Command gives $3$ families of solutions. Twist the $a_1$ back and include them as a parameter $p$, we get the families of solutions $\mathcal{C},\, \mathcal{D},\, \mathcal{E}$ below.
\end{enumerate}

\begin{example}
Let $\mathcal{A}=\{\mathcal{A}(p): p\neq0\}$, where $\mathcal{A}(p)=k\langle x_1, x_2\rangle/(f_1, f_2, f_3)$
is the $\Z^2$-graded algebra with relations
{\small\begin{eqnarray*}
\begin{split}
&f_1=x_2x_1^2-p^2x_1^2x_2,\\
&f_2=x_2^2x_1x_2x_1+px_2x_1x_2^2x_1-p^3x_1x_2^2x_1x_2-p^4x_1x_2x_1x_2^2,\\
&f_3=x_2^4x_1-p^4x_1x_2^4.
\end{split}
\end{eqnarray*}}
\end{example}

\begin{enumerate}
\item $\mathcal{A}(p)\in \mathcal{A}$ is isomorphic to a twisting $\Z^2$-graded algebra of $\mathcal{A}(q)$.
\item If $\mathcal{A}(p)\in\mathcal{A}$, then the element $z:=x_2^3x_1+px_2^2x_1x_2+p^2x_2x_1x_2^2+p^3x_1x_2^3$ is regular
normal in $\mathcal{A}(p)$ and the factor algebra
$\mathcal{A}(p)/(z)$ is isomorphic to $\mathbf{A}(\mathbf{p})$ in \cite{LPWZ} with $\mathbf{p}=p$.
\item Algebras $\mathcal{A}(p)\in\mathcal{A}$ are AS-regular of global dimension $5$. They are all strongly noetherian, Auslander regular and Cohen-Macaulay.
\end{enumerate}

\begin{example}
Let $\mathcal{B}=\{\mathcal{B}(p): p\neq0\}$, where $\mathcal{B}(p)=k\langle x_1, x_2\rangle/(f_1, f_2, f_3)$
is the $\Z^2$-graded algebra with relations
{\small\begin{eqnarray*}
\begin{split}
&f_1=x_2x_1^2-p^2x_1^2x_2,\\
&f_2=x_2^2x_1x_2x_1+px_2x_1x_2^2x_1-p^3x_1x_2^2x_1x_2-p^4x_1x_2x_1x_2^2,\\
&f_3=x_2^4x_1+p^4x_1x_2^4.
\end{split}
\end{eqnarray*}}
\end{example}

\begin{enumerate}
\item $\mathcal{B}(p)\in \mathcal{B}$ is isomorphic to a twisting $\Z^2$-graded algebra of $\mathcal{B}(1)$.
\item If $\mathcal{B}(p)\in\mathcal{B}$, denote $z_1:=x_1^2,\ z_2:=x_2^4,\ z_3:= (x_2x_1+px_1x_2)^2,\
z_4:=x_2^3x_1+px_2^2x_1x_2+p^2x_2x_1x_2^2+p^3x_1x_2^3$, and $z_5:=x_2^2x_1+p^2x_1x_2^2$, then $\{z_1, z_2, \cdots, z_5\}$ is a
normal sequence in $\mathcal{B}(p)$ such that the factor algebra
$\mathcal{B}(p)/(z_1, z_2, z_3, z_4, z_5)$ is finite dimensional.
\item Algebras $\mathcal{B}(p)\in\mathcal{B}$ are AS-regular of global dimension $5$. They are all strongly noetherian, Auslander regular and Cohen-Macaulay.
\end{enumerate}

\begin{example}
Let $\mathcal{C}=\{\mathcal{C}(p,j): p\neq0,\; j^2+j+1=0\}$, where
$\mathcal{C}(p,j)=k\langle x_1, x_2\rangle/(f_1, f_2, f_3)$
is the $\Z^2$-graded algebra with relations
{\small\begin{eqnarray*}
\begin{split}
&f_1=x_2x_1^2+ px_1x_2x_1+p^2x_1^2x_2,\\
&f_2=x_2^2x_1x_2x_1-p^2x_2x_1x_2x_1x_2+p^2(1+j)x_1x_2^3x_1
   + p^4x_1x_2x_1x_2^2+p^5(1-j)x_1^2x_2^3,\\
&f_3=x_2^4x_1-pjx_2^3x_1x_2-p^3jx_2x_1x_2^3-p^4(1+j)x_1x_2^4.
\end{split}
\end{eqnarray*}}
\end{example}

\begin{enumerate}
\item $\mathcal{C}(p, j)\in \mathcal{C}$ is isomorphic to a twisting $\Z^2$-graded algebra of $\mathcal{C}(1, j)$.
\item If $\mathcal{C}(p, j)\in\mathcal{C}$, then the element $z:=x_2^3x_1-p^3jx_1x_2^3$ is
regular normal in $\mathcal{C}(p, j)$ and the factor algebra
$\mathcal{C}(p, j)/(z)$ is isomorphic to $\mathbf{C}(\mathbf{p},\mathbf{j})$ in \cite{LPWZ} with
$\mathbf{p}=p,\ \mathbf{j}=-j$.
\item Algebras $\mathcal{C}(p, j)\in\mathcal{C}$ are AS-regular of global dimension $5$. They are all strongly noetherian, Auslander regular and Cohen-Macaulay.
\end{enumerate}

\begin{example}
Let $\mathcal{D}=\{\mathcal{D}(p): p\neq0\}$, where $\mathcal{D}(p)=k\langle x_1, x_2\rangle/(f_1, f_2, f_3)$
is the $\Z^2$-graded algebra with relations
{\small\begin{eqnarray*}
\begin{split}
&f_1=x_2x_1^2-2px_1x_2x_1+p^2x_1^2x_2,\\
&f_2=x_2^2x_1x_2x_1-3px_2x_1x_2^2x_1+2p^2x_2x_1x_2x_1x_2+2p^2x_1x_2^3x_1
   -3p^3x_1x_2^2x_1x_2+p^4x_1x_2x_1x_2^2,\\
&f_3=x_2^4x_1-4px_2^3x_1x_2+6p^2x_2^2x_1x_2^2-4p^3x_2x_1x_2^3+p^4x_1x_2^4.
\end{split}
\end{eqnarray*}}
\end{example}

\begin{enumerate}
\item $\mathcal{D}(p)\in \mathcal{D}$ is isomorphic to a twisting $\Z^2$-graded algebra of $\mathcal{D}(1)$.
\item If $\mathcal{D}(p)\in\mathcal{D}$, then the element $z:=x_2^3x_1-3px_2^2x_1x_2+3p^2x_2x_1x_2^2-p^3x_1x_2^3$ is
regular normal in $\mathcal{D}(p)$ and the factor algebra
$\mathcal{D}(p)/(z)$ is isomorphic to $\mathbf{D}(\mathbf{v},\mathbf{p})$ in \cite{LPWZ} with
$\mathbf{v}=-2p,\ \mathbf{p}=-p$.
\item $\mathcal{D}(1)$ is the enveloping algebra of the $\Z^2$-graded Lie algebra $\mathfrak{g}={\rm Lie}(x_1, x_2)/(f_1, f_2, f_3)$, where ${\rm Lie}(x_1, x_2)$ is the free Lie algebra on $\{x_1, x_2\}$ with grading given by
$\deg(x_1)=(1, 0)$, $\deg(x_2)=(0, 1)$, and $f_1=[[x_2 x_1] x_1]$, $f_2=[[x_2 [x_2 x_1]] [x_2 x_1]]$,\, $f_3=[x_2 [x_2 [x_2 [x_2 x_1]]]]$.
\item Algebras $\mathcal{D}(p)\in\mathcal{D}$ are AS-regular of global dimension $5$. They are all strongly noetherian, Auslander regular and Cohen-Macaulay.
\end{enumerate}

\begin{example}
Let $\mathcal{E}=\{\mathcal{E}(p, j): p\neq0,\; j^4+j^3+j^2+j+1=0\}$, where $\mathcal{E}(p, j)=k\langle x_1, x_2\rangle/(f_1, f_2, f_3)$
is the $\Z^2$-graded algebra with relations
{\small\begin{eqnarray*}
\begin{split}
&f_1=x_2x_1^2+px_1x_2x_1+p^2(j^2+j^3)^2x_1^2x_2,\\
&f_2=x_2^2x_1x_2x_1+p(j^2+j^3)^2x_2x_1x_2^2x_1-p^2(j^2+j^3)x_2x_1x_2x_1x_2\\
 & \hskip10mm +p^2(2+2j-j^3)x_1x_2^3x_1+p^3(3+5j+3j^2)x_1x_2^2x_1x_2\\
 &  \hskip10mm +p^4(4+8j+7j^2+2j^3)x_1x_2x_1x_2^2+p^5(2+8j+10j^2+5j^3)x_1^2x_2^3,\\
&f_3=x_2^4x_1+p(1+j)x_2^3x_1x_2+p^2(1+j)^2x_2^2x_1x_2^2\\
&  \hskip10mm  +p^3(1+j)^3x_2x_1x_2^3+p^4(3j+5j^2+3j^3)x_1x_2^4.
\end{split}
\end{eqnarray*}}
\end{example}

\begin{enumerate}
\item $\mathcal{E}(p, j)\in \mathcal{E}$ is isomorphic to a twisting $\Z^2$-graded algebra of $\mathcal{E}(1, j)$.
\item If $\mathcal{E}(p, j)\in\mathcal{E}$, then the element $z:=x_2^3x_1+p(j^2+j^3)^2x_2^2x_1x_2-p^2(j^2+j^3)^3x_2x_1x_2^2 -p^3(j^2+j^3)^3x_1x_2^3$
    is regular normal in $\mathcal{E}(p, j)$ and the factor algebra
$\mathcal{E}(p, j)/(z)$ is isomorphic to $\mathbf{D}(\mathbf{v},\mathbf{p})$ in \cite{LPWZ} with
$\mathbf{v}=p,\ \mathbf{p}=-p(j^2+j^3)$.
\item Algebras $\mathcal{E}(p, j)\in\mathcal{E}$ are AS-regular of global dimension $5$. They are all strongly noetherian, Auslander regular and Cohen-Macaulay.
\end{enumerate}

\section{Classification of type (3, 4, 7)}

Next we turn to the 5-dimensional AS-regular algebras of type (3, 4, 7), these are determined by three defining relations of the (total) degrees 3, 4, 7, respectively. One class of AS-regular algebras, denoted by $\mathcal{F}$, is obtained in this case.

Readily the minimal free resolution of the trivial module $k$ over $A^{\rm gr}$ is of the form
\begin{eqnarray*}
0
\to A^{\rm gr}(-12)
\to A^{\rm gr}(-11)^2
\to A^{\rm gr}(-5)\oplus A^{\rm gr}(-8)\oplus A^{\rm gr}(-9)\hskip25mm \\
\to
A^{\rm gr}(-3)\oplus A^{\rm gr}(-4)\oplus A^{\rm gr}(-7)
\to A^{\rm gr}(-1)^2
\to A^{\rm gr}
\to k_A
\to 0.
\end{eqnarray*}
Thus
$$
H_{(R^0)^{\rm gr}}(t)-H_{R^{\rm gr}}(t)=t^3+\sum_{n\geq4}r_nt^n
$$
and so $G$ has exactly one relation of total degree $3$,
which is $f_1$. Since $\deg_1(f_1) \geq \deg_2(f_1)$, one has
\begin{eqnarray}
\label{type-(3,4,7)-1}
\deg(f_1)=(2,1)\ \text{ with }\ \lw(f_1)=x_2x_1^2,\nonumber
\end{eqnarray}
which gives rise to
$$
H_{(R^1)^{\rm gr}}(t)-H_{R^{\rm gr}}(t)=t^4+\sum_{n\geq5}r_nt^n.
$$
So $G$ has exactly one relation of total degree $4$, which is $f_2$, and $G_{min}\supseteq\{f_1, f_2\}$. By checking normal words of degree
$(3, 1),(2, 2),(1, 3)$ modulo $G^1$, respectively, one has
$$
\deg(f_2)=(2, 2) \text{ with }
\lw(f_2)=x_2x_1x_2x_1\quad \text{or} \quad \deg(f_2)=(1, 3)
\text{ with } \lw(f_2)=x_2^3x_1.
$$

\begin{lemma}
$\deg(f_2)=(1, 3)$ with $\lw(f_2)=x_2^3x_1$, $\deg(f_3)=(3, 4)$ and $G_{min}=\{f_1, f_2, f_3\}$.
\end{lemma}

\begin{proof}
If $\deg(f_2)=(2, 2)$ with $\lw(f_2)=x_2x_1x_2x_1$, then
$$
H_{(R^2)^{\rm gr}}(t)-H_{R^{\rm gr}}(t)=t^6+\sum_{n\geq7}r_nt^n,
$$
and hence $G$ has exactly one relation of total degree $6$, which is $f_3$.
Consider the unique composition \linebreak $S(f_2, f_2)[1, x_2x_1, x_2x_1, 1]$
of relations in $G^2$ of total degree $6$.
Since $(G^2)_{\leq6}=(G)_{\leq6}$, the remainder of the composition should be a nonzero scalar multiple of $f_3$ by Lemma \ref{truncated diamond lemma}. So $\deg(f_3)=(3, 3)$. However, by checking all normal
words of degree $(3,3)$ modulo $G^2$, any choice of $\lw(f_3)$ will contradict that $A$
is a domain. Now we reach the only possibility that $\deg(f_2)=(1, 3)$ with $\lw(f_2)=x_2^3x_1$,
which follows that
$$
H_{(R^2)^{\rm gr}}(t)-H_{R^{\rm gr}}(t)=t^7+\sum_{n\geq8}r_nt^n.
$$
So $G$ has exactly one relation of total degree $7$, which is $f_3$.
If $\deg(f_3)=\alpha\neq(3, 4)$, then by checking all normal words of degree $\alpha$ modulo $G^2$,
any choice of $\lw(f_3)$ would contradict that $A$ is a domain. So $\lw(f_3)=(3, 4)$.
The fact $G_{min}=\{f_1, f_2, f_3\}$ follows from the trivial observation that
the relations in $G$ other than $f_1, f_2, f_3$ are all of total degree $\geq8$.
We finish the proof.
\end{proof}

\begin{proposition}
\label{resolution-type-(3,4,7)}
The minimal free resolution of $k_A$ is of the form
\begin{equation*}
0 \to
\begin{array}[t]{c}
A\\[-0.5em] \text{\tiny\rm (5, 7)}
\end{array}
\to
\begin{array}[t]{c}
A^2\\[-0.5em] \text{\tiny \rm (5, 6)}\\[-0.7em] \text{\tiny \rm (4, 7)}
\end{array}
\to
\begin{array}[t]{c}
A^3\\[-0.5em] \text{\tiny \rm (2, 3)}\\[-0.7em] \text{\tiny \rm (4, 4)}\\[-0.7em] \text{\tiny \rm (3, 6)}
\end{array}
\to
\begin{array}[t]{c}
A^3\\[-0.5em] \text{\tiny \rm (2, 1)}\\[-0.7em] \text{\tiny \rm (1, 3)}\\[-0.7em] \text{\tiny \rm (3, 4)}
\end{array}
\to
\begin{array}[t]{c}
A^2\\[-0.5em] \text{\tiny \rm (1, 0)}\\[-0.7em] \text{\tiny \rm (0, 1)}
\end{array}
\to
\begin{array}[t]{c}
A\\[-0.5em] \text{\tiny \rm (0, 0)}
\end{array}
\to k_A
\to 0.
\end{equation*}
\end{proposition}

\begin{proof}
It is also obvious by applying Lemma \ref{minimal resolution}.
\end{proof}

\begin{lemma}
$\lw(f_3)=x_2^2x_1x_2x_1x_2x_1$.
\end{lemma}

\begin{proof}

By checking all normal words of degree $(3, 4)$ modulo $G^2$, one has
$$
\lw(f_3)\in\{x_2x_1x_2x_1x_2x_1x_2,\ x_2x_1x_2x_2x_1x_2x_1,\
x_2x_1x_2x_1x_2^2x_1,\ x_2^2x_1x_2x_1x_2x_1\}.
$$
$\lw(f_3)=x_2x_1x_2x_1x_2x_1x_2$ causes the contradiction
$$
H_{R^3}(t_1, t_2)-H_{R}(t_1, t_2)=-t_1^4t_2^4+
\sum_{|\alpha|\geq8}r_{\alpha}t_1^{\alpha_1}t_2^{\alpha_2},
$$
also, $\lw(f_3)\in\{x_2x_1x_2x_2x_1x_2x_1,\ x_2x_1x_2x_1x_2^2x_1\}$
gives the contradiction
\begin{eqnarray*}
H_{R^3}(t_1,t_2)-H_{R}(t_1,t_2)=-t_1^4t_2^6+
\sum_{|\alpha|\geq11}r_{\alpha}t_1^{\alpha_1}t_2^{\alpha_2},
\end{eqnarray*}
we reach the only possibility $\lw(f_3)=x_2^2x_1x_2x_1x_2x_1$ and finish the proof.
\end{proof}

\begin{proposition}
\label{type-(3,4,7)}
$G=\{f_1, f_2, f_3, f_4\}$ and $G_{min}=\{f_1, f_2, f_3\}$ with
{\small\begin{eqnarray*}
\begin{split}
&f_1=x_2x_1^2+ a_1x_1x_2x_1+a_2x_1^2x_2,\\
&f_2=x_2^3x_1+b_1x_2^2x_1x_2+b_2x_2x_1x_2^2+b_3x_1x_2^3,\\
&f_3=x_2^2x_1x_2x_1x_2x_1+c_1x_2x_1x_2^2x_1x_2x_1+c_2x_2x_1x_2x_1x_2^2x_1 +c_3x_2x_1x_2x_1x_2x_1x_2+c_4x_1x_2^2x_1x_2^2x_1
\\& \hskip10mm
+c_5x_1x_2^2x_1x_2x_1x_2+c_6x_1x_2x_1x_2^2x_1x_2 +c_7x_1x_2x_1x_2x_1x_2^2
+c_8x_1^2x_2^2x_1x_2^2+c_9x_1^2x_2x_1x_2^3+c_{10}x_1^3x_2^4,\\
&f_4=x_2^2x_1x_2^2x_1x_2x_1+d_1x_2^2x_1x_2x_1x_2^2x_1+d_2x_2x_1x_2^2x_1x_2^2x_1 +d_3x_2x_1x_2^2x_1x_2x_1x_2+d_4x_2x_1x_2x_1x_2^2x_1x_2
\\& \hskip10mm
+d_5x_2x_1x_2x_1x_2x_1x_2^2+d_6x_1x_2^2x_1x_2^2x_1x_2 +d_7x_1x_2^2x_1x_2x_1x_2^2+d_8x_1x_2x_1x_2^2x_1x_2^2
\\& \hskip10mm
+d_9x_1x_2x_1x_2x_1x_2^3+d_{10}x_1^2x_2^2x_1x_2^3
+d_{11}x_1^2x_2x_1x_2^4+d_{12}x_1^3x_2^5,
\end{split}
\end{eqnarray*}}
where $a_i, b_i, c_i, d_i\in k$. Moreover, one has $a_1=b_1=b_2=0$, $c_1\neq0$ and $a_2^3=b_3^2\neq0$.
\end{proposition}

\begin{proof}
By previous discussion, we have $G_{min}=\{f_1, f_2, f_3\}$ and $f_1, f_2, f_3$ are of the given form by checking normal words of corresponding degrees modulo $G^3$ (which are same modulo $G$). The equations $a_1=b_1=b_2=0$ and $a_2^3=b_3^2$ follow from that the coefficients of the remainder of the composition $S(f_2, f_1)[1, x_1, x_2^2, 1]$ modulo $G^3$ are all zero. The inequality $b_3^2\neq0$ follows from that $A$ is a domain. It remains to show that $c_1\neq0$, and $G=\{f_1, f_2, f_3, f_4\}$ with $f_4$ of the giving form. Since
$$
H_{R^3}(t_1,t_2)-H_{R}(t_1,t_2)=t_1^3t_2^5+
\sum_{|\alpha|\geq11}r_{\alpha}t_1^{\alpha_1}t_2^{\alpha_2},
$$
$G$ has exactly one relation of total degree $8$,
which is $f_4$ of degree $(3,5)$. By Checking all normal words
of degree $(3, 5)$ modulo $G^3$, one has $\lw(f_4)\in \{x_2x_1x_2^2x_1x_2^2x_1,\ x_2^2x_1x_2x_1x_2^2x_1,\
x_2^2x_1x_2^2x_1x_2x_1\}$. Since $\lw(f_4)=x_2x_1x_2^2x_1x_2^2x_1$ gives the contradiction
\begin{eqnarray*}
H_{R^4}(t_1, t_2)-H_{R}(t_1, t_2)=-t_1^3t_2^6+
\sum_{|\alpha|\geq10}r_{\alpha}t_1^{\alpha_1}t_2^{\alpha_2},\\
\end{eqnarray*}
and $\lw(f_4)=x_2^2x_1x_2x_1x_2^2x_1$ gives the contradiction
\begin{eqnarray*}
H_{R^4}(t_1, t_2)-H_{R}(t_1, t_2)=-t_1^4t_2^6+
\sum_{|\alpha|\geq10}t_1^{\alpha_1}t_2^{\alpha_2},
\end{eqnarray*}
we reach the only possibility $\lw(f_4)=x_2^2x_1x_2^2x_1x_2x_1$,
which follows that $H_{R^4}(t_1, t_2)=H_{R}(t_1, t_2)$ and thereby $G=G^4$.
Now by checking normal words of degree $(3,5)$ modulo $G$, one gets that $f_4$ is of the given form. Applying Lemma \ref{truncated diamond lemma}, the remainder of the composition $S(f_2, f_3)[1, x_2x_1x_2x_1, x_2, 1]$ modulo $G^3$
is a non-zero scalar multiple of $f_4$. It follows immediately that $b_1\neq c_1$ and then we finish the proof.
\end{proof}

Now we search for the possible solutions of the coefficients of $f_i$ in Proposition \ref{type-(3,4,7)}. Firstly note that there are six compositions of polynomials in $G$, namely
{\small\begin{eqnarray*}
\begin{array}{llll}
&S(f_2, f_1)[1, x_1, x_2^2, 1], &S(f_2, f_3)[1, x_2x_1x_2x_1, x_2, 1],
&S(f_2, f_4)[1, x_2^2x_1x_2x_1, x_2, 1],\\
&S(f_3, f_1)[1, x_1, x_2^2x_1x_2x_1, 1], &S(f_4, f_1)[1, x_1, x_2^2x_1x_2^2x_1, 1], &S(f_4, f_3)[1, x_2x_1, x_2^2x_1, 1].
\end{array}
\end{eqnarray*}}
A simple Program in Maple calculate the remainder of these compositions modulo $G$ in $k\langle x_1, x_2\rangle$ with results denoted by $r_1, r_2, \cdots, r_6$, respectively.

Apply Lemma \ref{truncated diamond lemma}, as a necessary condition for $G$ to be Gr\"{o}bner, all the coefficients of $r_i$'s should be zero.  We use Maple to help solve the possible solutions to the system of equations given by setting the coefficients of $r_i$'s to $0$, together with $a_1=b_1=b_2=0$, $c_1\neq0$ and $a_2^3=b_3^2\neq0$. By Lemma \ref{twisting}, we can assume $c_1=1$. Then Maple Solving Command gives $1$ family of solutions. Twist the $c_1$ back and include it as a parameter $p$, we get a family of solutions $\mathcal{F}$ below.

\begin{example}
Let $\mathcal{F}=\{\mathcal{F}(p, q): p\neq0\}$, where $\mathcal{F}(p,q)=k\langle x_1, x_2\rangle/(f_1, f_2, f_3)$
is the $\Z^2$-graded algebra with relations
{\small\begin{eqnarray*}
\begin{split}
&f_1=x_2x_1^2-p^2x_1^2x_2,\\
&f_2=x_2^3x_1-p^3x_1x_2^3,\\
&f_3=x_2^2x_1x_2x_1x_2x_1+px_2x_1x_2^2x_1x_2x_1+p^2x_2x_1x_2x_1x_2^2x_1 -p^4x_1x_2^2x_1x_2x_1x_2\\
    & \hskip10mm -p^5x_1x_2x_1x_2^2x_1x_2-p^6x_1x_2x_1x_2x_1x_2^2 -p^8qx_1^2x_2x_1x_2^3+p^9qx_1^3x_2^4.
\end{split}
\end{eqnarray*}}
\end{example}

\begin{enumerate}
\item
$\mathcal{F}(p, q)\in \mathcal{F}$ is isomorphic to a twisting $\Z^2$-graded algebra of
$\mathcal{F}(1, q)$.
\item If $\mathcal{F}(p, q)\in\mathcal{F}$, then the sequence
$z_1:=x_1^2,\ z_2:=x_2^3,\ z_3:=(x_2x_1+px_1x_2)^3,\ z_4:=(x_2^2x_1+px_2x_1x_2+p^2x_1x_2^2)^2,\
z_5:=x_2^2x_1x_2x_1+p\theta x_2x_1x_2^2x_1+p^3x_1x_2^2x_1x_2+p^4\theta x_1x_2x_1x_2^2$,
where $\theta\in k$ satisfying that $\theta^2-\theta+1=0$,  is a normal sequence in
$\mathcal{F}(p, q)$ such that the factor algebra $\mathcal{F}(p, q)/(z_1, z_2, z_3, z_4, z_5)$ is finite dimensional.
\item Algebras $\mathcal{F}(p, j)\in\mathcal{F}$ are AS-regular of global dimension $5$. They are all strongly noetherian, Auslander regular and Cohen-Macaulay.
\end{enumerate}

The algebra $\mathcal{F}(1, 0)$ is the extremal algebra considered in \cite{FV} which has a Hilbert series
not occurring for enveloping algebras of graded Lie algebras generated in degree one. The normal sequence $z_1, z_2,\cdots, z_6$ for $\mathcal{F}(1, 0)$ was firstly given in \cite[Proposition 4.1]{W}, which turned out to be a normal sequence for $\mathcal{F}(1,q)$ by a straightforward check. Then apply Lemma \ref{twisting}, we obtain the normal sequence for general algebras $\mathcal{F}(p, q)$.

\section{Classification of type (4, 4, 4, 5)}

Now we proceed to discuss the case of 5-dimensional AS-regular $\Z^2$-graded algebras of type $(4, 4, 4, 5)$, these are determined by four defining relations of the (total) degrees 4, 4, 4, 5, respectively. It turns out that there is no such AS-regular algebra in the context of $\Z^2$-grading.

The minimal free resolution of the trivial module $k$ over $A^{\rm gr}$ is of the form
\begin{eqnarray*}
0
\to A^{\rm gr}(-10)
\to A^{\rm gr}(-9)^2
\to A^{\rm gr}(-5)\oplus A^{\rm gr}(-6)^3 \\
\to
A^{\rm gr}(-4)^3\oplus A^{\rm gr}(-5)
&\to A^{\rm gr}(-1)^2
\to A^{\rm gr}
\to k_A
\to 0.
\end{eqnarray*}
So
$$
H_{(R^0)^{\rm gr}}(t)-H_{R^{\rm gr}}(t)=3t^4+\sum_{n\geq5}r_nt^n
$$
and then $G$ has exactly three relations, say $f_1, f_2, f_3$, of total degree $4$. Hence $G_{min}\supseteq\{f_1, f_2, f_3\}$,
the degrees of $f_1, f_3$ are both contained in
$\{(3, 1), (2, 2), (1, 3)\}$ and
$$
\deg(f_2)=(2, 2)\quad \text{with}\quad \lw(f_2)\in \{x_2^2x_1^2, \ x_2x_1x_2x_1,\ x_2x_1^2x_2\}.
$$

\begin{lemma}
It is impossible that $\deg(f_1)=\deg(f_3)=(2, 2)$.
\end{lemma}

\begin{proof}
Otherwise, all the degrees of $f_1, f_2, f_3$ are $(2, 2)$,
by Lemma \ref{minimal resolution}, there's no candidate degree for the fourth relation in $G_{min}$ of total degree $5$. The result follows immediately.
\end{proof}

\begin{lemma}
\label{lemma2-type-(4,4,4,5)}
$\deg(f_1)=(3,1)$ with $\lw(f_1)=x_2x_1^3$ and $\deg(f_3)=(1, 3)$ with $\lw(f_3)=x_2^3x_1$.
\end{lemma}

\begin{proof}
If $\deg(f_3)=(2, 2)$, then $\deg(f_1)=(3, 1)$ with $\lw(f_1)=x_2x_1^3$.
By Lemma \ref{minimal resolution}
the only possible minimal free resolution of $k_A$ is of the form:
\begin{equation*}
0 \to
\begin{array}[t]{c}
A\\[-0.5em] \text{\tiny\rm (6, 4)}
\end{array}
\to
\begin{array}[t]{c}
A^2\\[-0.5em] \text{\tiny \rm (6, 3)}\\[-0.7em] \text{\tiny \rm (5, 4)}
\end{array}
\to
\begin{array}[t]{c}
A^3\\[-0.5em] \text{\tiny \rm (3, 2)}\\[-0.7em] \text{\tiny \rm (4, 2)}\\[-0.7em] \text{\tiny \rm (4, 2)}\\[-0.7em] \text{\tiny \rm (3, 3)}
\end{array}
\to
\begin{array}[t]{c}
A^3\\[-0.5em] \text{\tiny \rm (3, 1)}\\[-0.7em] \text{\tiny \rm (2, 2)}\\[-0.7em] \text{\tiny \rm (2, 2)}\\[-0.7em] \text{\tiny \rm (3, 2)}
\end{array}
\to
\begin{array}[t]{c}
A^2\\[-0.5em] \text{\tiny \rm (1, 0)}\\[-0.7em] \text{\tiny \rm (0, 1)}
\end{array}
\to
\begin{array}[t]{c}
A\\[-0.5em] \text{\tiny \rm (0, 0)}
\end{array}
\to k_A
\to 0.
\end{equation*}

Since $A$ is a domain, there are three possibilities:

(a) $\lw(f_2)=x_2x_1^2x_2$, $\lw(f_3)=x_2x_1x_2x_1$.
Then
$$
H_{R^3}(t_1, t_2)-H_{R}(t_1, t_2)=-t_1^4t_2^2+t_1^3t_2^3+
\sum_{|\alpha|\geq7}r_{\alpha}t_1^{\alpha_1}t_2^{\alpha_2}.
$$

(b) $\lw(f_2)=x_2x_1^2x_2$, $\lw(f_3)=x_2^2x_1^2$.
Then
$$
H_{R^3}(t_1, t_2)-H_{R}(t_1, t_2)=t_1^3t_2^2+t_1^2t_2^3+
\sum_{|\alpha|\geq6}r_{\alpha}t_1^{\alpha_1}t_2^{\alpha_2}.
$$
So $G$ has exactly two relations of total degree $5$,
which are $f_4$ of degree $(3, 2)$ and $f_5$ of degree $(2, 3)$.
One easily gets $\lw(f_4)=x_2x_1x_2x_1^2$ by checking all normal
words of degree $(3, 2)$ modulo $G^3$. To find the word
$\lw(f_5)$, one should consider the unique composition
$
S(f_3, f_2)[1, x_2, x_2, 1]
$
of relations in $G^3$ of degree $(2, 3)$, whose remainder modulo $G^3$ is a nonzero scalar multiple of $f_5$ in $k\langle x_1, x_2\rangle$ by Lemma \ref{truncated diamond lemma}. In fact,
the composition is of the form
$
px_2x_1x_2^2x_1+gx_2,
$
where $p\in k$ and  $g\in k\langle x_1, x_2\rangle_{(2, 2)}$ with $\lw(gx_2)<x_2x_1x_2^2x_1$.
So one gets $p\neq0$ and
$\lw(f_5)=x_2x_1x_2^2x_1$. Thus
$$
H_{R^5}(t_1, t_2)-H_R(t_1, t_2)=-t_1^4t_2^2+t_1^3t_2^3
+\sum_{|\alpha|\geq7}r_\alpha t_1^{\alpha_1}t_2^{\alpha_2}.
$$

(c) $\lw(f_2)=x_2x_1x_2x_1$, $\lw(f_3)=x_2^2x_1^2$.
Then
$$
H_{R^3}(t_1, t_2)-H_{R}(t_1, t_2)=t_1^3t_2^2+
\sum_{|\alpha|\geq6}r_{\alpha}t_1^{\alpha_1}t_2^{\alpha_2}.
$$
So $G$ has exactly one element of total degree $5$, which is $f_4$ of degree $(3, 2)$.
Checking all normal words of degree $(3, 2)$ modulo $G^3$,
one gets $\lw(f_4)=x_2x_1^2x_2x_1$ and then
$$
H_{R^4}(t_1, t_2)-H_R(t_1, t_2)=-t_1^4t_2^2+t_1^3t_2^3
+\sum_{|\alpha|\geq7}r_\alpha t_1^{\alpha_1}t_2^{\alpha_2}.
$$

All these three possibilities cause a contradiction and so $\deg(f_3)=(3, 1)$.
The possibility of $\deg(f_1)=(2, 2)$ can be excluded symmetrically and thereby the
result follows.
\end{proof}

\begin{proposition}
\label{type-(4,4,4,5)}
There is no such an algebra of type $(4, 4, 4, 5)$.
\end{proposition}

\begin{proof}
Suppose such an algebra $A$ exists. Then by switching and Lemma \ref{minimal resolution} one can always assume that the minimal free resolution of $k_A$ is of the form:
\begin{equation*}
0 \to
\begin{array}[t]{c}
A\\[-0.5em] \text{\tiny\rm (5, 5)}
\end{array}
\to
\begin{array}[t]{c}
A^2\\[-0.5em] \text{\tiny \rm (5, 4)}\\[-0.7em] \text{\tiny \rm (4, 5)}
\end{array}
\to
\begin{array}[t]{c}
A^3\\[-0.5em] \text{\tiny \rm (2, 3)}\\[-0.7em] \text{\tiny \rm (4, 2)}\\[-0.7em] \text{\tiny \rm (3, 3)}\\[-0.7em] \text{\tiny \rm (2, 4)}
\end{array}
\to
\begin{array}[t]{c}
A^3\\[-0.5em] \text{\tiny \rm (3, 1)}\\[-0.7em] \text{\tiny \rm (2, 2)}\\[-0.7em] \text{\tiny \rm (1, 3)}\\[-0.7em] \text{\tiny \rm (3, 2)}
\end{array}
\to
\begin{array}[t]{c}
A^2\\[-0.5em]  \text{\tiny \rm (1, 0)}\\[-0.7em] \text{\tiny \rm (0, 1)}
\end{array}
\to
\begin{array}[t]{c}
A\\[-0.5em] \text{\tiny \rm (0, 0)}
\end{array}
\to k_A
\to 0.
\end{equation*}
If $\lw(f_2)=x_2x_1^2x_2$ or $\lw(f_2)=x_2x_1x_2x_1$, then one gets a contradiction
$$
H_{R^3}(t_1, t_2)-H_{R}(t_1, t_2)=t_1^3t_2^2-t_1^2t_2^3+
\sum_{|\alpha|\geq6}r_{\alpha}t_1^{\alpha_1}t_2^{\alpha_2}.
$$
So it remains to exclude the possibility of $\lw(f_2)=x_2^2x_1^2$. Otherwise,
$$
H_{R^3}(t_1, t_2)-H_{R}(t_1, t_2)=2t_1^3t_2^2+
\sum_{|\alpha|\geq6}r_{\alpha}t_1^{\alpha_1}t_2^{\alpha_2}
$$
and then $G$ contains exactly two elements, which are $f_4,\ f_5$ of degree $(3, 2)$. By checking all normal words of degree $(3, 2)$ modulo $G^3$, one gets
$\lw(f_4)=x_2x_1^2x_2x_1,\ \lw(f_5)=x_2x_1x_2x_1^2$ and hence
$$
H_{R^5}(t_1, t_2)-H_{R}(t_1, t_2)=-t_1^2t_2^4+
\sum_{|\alpha|\geq7}r_{\alpha}t_1^{\alpha_1}t_2^{\alpha_2},
$$
which is impossible.  We complete the proof.
\end{proof}

\section{Classification of type (4, 4, 4, 5, 5)}

In this section, we devote to classify 5-dimensional AS-regular $\Z^2$-graded algebras of type $(4, 4, 4, 5, 5)$, these are determined by five defining relations of the (total) degrees 4, 4, 4, 5, 5, respectively. We get a class of AS-regular algebras denoted by $\mathcal{G}$, and hence give an answer to the question of \cite{FV}.

The minimal free resolution of the trivial module $k$ over $A^{\rm gr}$ is of the form
\begin{eqnarray*}
0
\to A^{\rm gr}(-10)
\to A^{\rm gr}(-9)^2
\to A^{\rm gr}(-5)^2  \oplus A^{\rm gr}(-6)^3\hskip27mm \\
\to
A^{\rm gr}(-4)^3\oplus A^{\rm gr}(-5)^2
\to A^{\rm gr}(-1)^2
\to A^{\rm gr}
\to k_A
\to 0.
\end{eqnarray*}
So
$$
H_{(R^0)^{\rm gr}}(t)-H_{R^{\rm gr}}(t)=3t^4+\sum_{n\geq5}r_nt^n,
$$
and then $G$ has exactly $3$ relations of total degree $4$,
which are $f_1, f_2, f_3$. Hence
$G_{min}\supseteq\{f_1, f_2, f_3\}$, the degrees of $f_1, f_3$ are both contained in $\{(3, 1), (2, 2), (1, 3)\}$ and
$$
\deg(f_2)=(2, 2)\quad \text{with}\quad \lw(f_2)\in \{x_2^2x_1^2, x_2x_1x_2x_1, x_2x_1^2x_2\}.
$$

\begin{lemma}
It is impossible that $\deg(f_1)=\deg(f_3)=(2, 2)$.
\end{lemma}

\begin{proof}
If all the degrees of $f_1, f_2, f_3$ are $(2, 2)$,
Lemma \ref{minimal resolution} concludes that there's no candidate degree for the
fourth relation in $G_{min}$ of total degree $5$. So the result follows.
\end{proof}

\begin{lemma}
$\deg(f_1)=(3, 1)$ with $\lw(f_1)=x_2x_1^3$ and $\deg(f_3)=(1, 3)$ with $\lw(f_3)=x_2^3x_1$.
\end{lemma}

\begin{proof}
If $\deg(f_3)=(2, 2)$, then $\deg(f_1)=(3, 1)$ with $\lw(f_1)=x_2x_1^3$.
By Lemma \ref{minimal resolution}, the only possible
minimal free resolution of $k_A$ is of the form:
\begin{equation*}
0 \to
\begin{array}[t]{c}
A\\[-0.5em] \text{\tiny\rm (6, 4)}
\end{array}
\to
\begin{array}[t]{c}
A^2\\[-0.5em] \text{\tiny \rm (6, 3)}\\[-0.7em] \text{\tiny \rm (5, 4)}
\end{array}
\to
\begin{array}[t]{c}
A^3\\[-0.5em] \text{\tiny \rm (3, 2)}\\[-0.7em] \text{\tiny \rm (3, 2)}\\[-0.7em]
\text{\tiny \rm (4, 2)}\\[-0.7em] \text{\tiny \rm (4, 2)}\\[-0.7em] \text{\tiny \rm (3, 3)}
\end{array}
\to
\begin{array}[t]{c}
A^3\\[-0.5em] \text{\tiny \rm (3, 1)}\\[-0.7em] \text{\tiny \rm (2, 2)}\\[-0.7em]
\text{\tiny \rm (2, 2)}\\[-0.7em] \text{\tiny \rm (3, 2)}\\[-0.7em] \text{\tiny \rm (3, 2)}
\end{array}
\to
\begin{array}[t]{c}
A^2\\[-0.5em] \text{\tiny \rm (1, 0)}\\[-0.7em] \text{\tiny \rm (0, 1)}
\end{array}
\to
\begin{array}[t]{c}
A\\[-0.5em] \text{\tiny \rm (0, 0)}
\end{array}
\to k_A
\to 0.
\end{equation*}
The result follows immediately from the same discussion as that
in the proof of Lemma \ref{lemma2-type-(4,4,4,5)}.
\end{proof}

\begin{proposition}
\label{resolution-type-(4,4,4,5,5)}
The minimal free resolution of $k_A$ is of the form
\begin{equation*}
0 \to
\begin{array}[t]{c}
A\\[-0.5em] \text{\tiny\rm (5, 5)}
\end{array}
\to
\begin{array}[t]{c}
A^2\\[-0.5em] \text{\tiny \rm (5, 4)}\\ [-0.7em]\text{\tiny \rm  (4, 5)}
\end{array}
\to
\begin{array}[t]{c}
A^3\\[-0.5em] \text{\tiny \rm (3, 2)}\\[-0.7em] \text{\tiny \rm (2, 3)}\\[-0.7em]
   \text{\tiny \rm (4, 2)}\\ [-0.7em]\text{\tiny \rm (3, 3)}\\[-0.7em] \text{\tiny \rm (2, 4)}
\end{array}
\to
\begin{array}[t]{c}
A^3\\[-0.5em] \text{\tiny \rm (3, 1)}\\ [-0.7em] \text{\tiny \rm  (2, 2)}\\[-0.7em]
   \text{\tiny \rm (1, 3)}\\ [-0.7em] \text{\tiny \rm  (3, 2)}\\[-0.7em] \text{\tiny \rm (2, 3)}
\end{array}
\to
\begin{array}[t]{c}
A^2\\[-0.5em] \text{\tiny \rm (1, 0)}\\ [-0.7em] \text{\tiny \rm  (0, 1)}
\end{array}
\to
\begin{array}[t]{c}
A\\[-0.5em] \text{\tiny \rm (0, 0)}
\end{array}
\to k_A
\to 0.
\end{equation*}
\end{proposition}

\begin{proof}
By a similar but simpler discussion as that in the proof of Proposition \ref{type-(4,4,4,5)},
one sees that, other than $f_1, f_2, f_3$, the remaining two relations in $G_{min}$ of total degree $5$ are of different degrees. The result then follows from Lemma \ref{minimal resolution}.
\end{proof}

\begin{lemma}
$\lw(f_2)=x_2^2x_1^2$.
\end{lemma}

\begin{proof}
If $\lw(f_2)=x_2x_1^2x_2$ (resp. $\lw(f_2)=x_2x_1x_2x_1$),
then one will get a contradiction from
\begin{eqnarray*}
H_{R^3}(t_1, t_2)-H_{R}(t_1, t_2)=-t_1^4t_2^2
+\sum_{|\alpha|\geq7}r_{\alpha}t_1^{\alpha_1}t_2^{\alpha_2}
\quad (\text{resp.}\
H_{R^3}(t_1, t_2)-H_{R}(t_1, t_2)=t_1^3t_2^3
+\sum_{|\alpha|\geq7}r_{\alpha}t_1^{\alpha_1}t_2^{\alpha_2}).
\end{eqnarray*}
In fact, the later tells that there's no relation in $G$ of total degree $5$.
Now we reach the only possibility $\lw(f_2)=x_2^2x_1^2$.
\end{proof}

\begin{proposition}
\label{type-(4,4,4,5,5)}
$G=G_{min}=\{f_1, f_2, f_3, f_4, f_5\}$ with
{\small\begin{eqnarray*}
\begin{split}
&f_1=x_2x_1^3+a_1x_1x_2x_1^2+a_2x_1^2x_2x_1+a_3x_1^3x_2,\\
&f_2=x_2^2x_1^2+b_1x_2x_1x_2x_1+b_2x_2x_1^2x_2+b_3x_1x_2^2x_1+b_4x_1x_2x_1x_2 +b_5x_1^2x_2^2,\\
&f_3=x_2^3x_1+c_1x_2^2x_1x_2+c_2x_2x_1x_2^2+c_3x_1x_2^3,\\
&f_4=x_2x_1x_2x_1^2+d_1x_2x_1^2x_2x_1+d_2x_1x_2x_1x_2x_1+d_3x_1x_2x_1^2x_2 +d_4x_1^2x_2^2x_1 + d_5x_1^2x_2x_1x_2+d_6x_1^3x_2^2,\\
&f_5=x_2^2x_1x_2x_1+e_1x_2x_1x_2^2x_1+e_2x_2x_1x_2x_1x_2+e_3x_2x_1^2x_2^2 +e_4x_1x_2^2x_1x_2 + e_5x_1x_2x_1x_2^2+e_6x_1^2x_2^3,
\end{split}
\end{eqnarray*}}
where all $a_i, b_i, c_i, d_i, e_i\in k$. Moreover, one has $a_3 c_3\neq0$ and $a_1=b_1=c_1$.
\end{proposition}

\begin{proof}
By previous discussions, we have $f_1, f_2, f_3$ are of the given form by checking normal words of corresponding degrees modulo $G^3$ (which are same modulo $G$). Since $A$ is a domain, we have $a_3c_3\neq0$. It remains to show
that $a_1= b_1=c_1$ and $G=G_{min}=\{f_1, f_2, f_3, f_4, f_5\}$ with $f_4, f_5$ of the giving form. Since
$$
H_{R^3}(t_1, t_2)-H_{R}(t_1, t_2)=t_1^3t_2^2+t_1^2t_2^3+
\sum_{|\alpha|\geq6}t_1^{\alpha_1}t_2^{\alpha_2},
$$
$G$ has exactly two relations of total degree $5$,
which are $f_4$ of degree $(3, 2)$ and $f_5$ of degree $(2, 3)$. By checking all normal words of degree
$(3, 2)$ and $(2, 3)$ modulo $G^3$, one has
$$
\lw(f_4)\in\{x_2x_1x_2x_1^2,\ x_2x_1^2x_2x_1\}
\quad \text{and}\quad \lw(f_5)\in\{x_2^2x_1x_2x_1,\ x_2x_1x_2^2x_1\}.
$$
Readily we have the following results:
\begin{enumerate}
\item
If $\lw(f_4)=x_2x_1^2x_2x_1$ and $\lw(f_5)=x_2x_1x_2^2x_1$, then
$$
H_{R^5}(t_1, t_2)-H_{R}(t_1, t_2)=-t_1^4t_2^2+t_1^3t_2^3-t_1^2t_2^4+
\sum_{|\alpha|\geq7}t_1^{\alpha_1}t_2^{\alpha_2}.
$$
\item
If $\lw(f_4)=x_2x_1^2x_2x_1$ and $\lw(f_5)=x_2^2x_1x_2x_1$, then
$$
H_{R^5}(t_1, t_2)-H_{R}(t_1, t_2)=-t_1^4t_2^2+
\sum_{|\alpha|\geq7}t_1^{\alpha_1}t_2^{\alpha_2}.
$$
\item
If $\lw(f_4)=x_2x_1x_2x_1^2$ and $\lw(f_5)=x_2x_1x_2^2x_1$, then
$$
H_{R^5}(t_1, t_2)-H_{R}(t_1, t_2)=-t_1^2t_2^4+
\sum_{|\alpha|\geq7}t_1^{\alpha_1}t_2^{\alpha_2}.
$$
\end{enumerate}

All of those are impossible by Lemma \ref{lemma-hilbert-compare}, and hence we reach the only possibility
$$
\lw(f_4)=x_2x_1x_2x_1^2\quad\text{and}\quad \lw(f_5)=x_2^2x_1x_2x_1,
$$
which follows that $H_{R^5}(t_1, t_2)=H_{R}(t_1, t_2)$ and thereby $G=G^5$. Now by checking normal words of corresponding degrees modulo $G$, one gets that $f_4, f_5$ are of the given form. We also have $G_{min}=G$ from the resolution type. Now it remains to show $a_1=b_1=c_1$. In fact, by Lemma \ref{truncated diamond lemma}, if the remainder of the composition
$S(f_2, f_1)[1, x_1, x_2, 1]$\, (resp. $S(f_3, f_2)[1, x_1, x_2, 1]$)
modulo $G^3$ is not zero, then it is a scalar multiple of $f_4$ (resp. $f_5$),
which contradicts that $G$ is minimal. Thereby $a_1=b_1=c_1$ and then we finish the proof.
\end{proof}

Now we look for the possible solutions of the coefficients of $f_i$ in Proposition \ref{type-(4,4,4,5,5)}. Firstly note that there are nine compositions of polynomials in $G$, namely
{\small\begin{eqnarray*}
\begin{array}{lllll}
&r_1=S(f_2,f_1)[1,x_1,x_2,1], &r_2=S(f_3,f_1)[1,x_1^2,x_2^2,1], &r_3=S(f_3,f_2)[1,x_1,x_2,1],\\ &r_4=S(f_3,f_4)[1,x_2x_1^2,x_2^2,1],
&r_5=S(f_3,f_5)[1,x_2x_1,x_2,1],&r_6=S(f_4,f_1)[1,x_1,x_2x_1,1],\\
&r_7=S(f_5,f_1)[1,x_1^2,x_2^2x_1,1],&r_8=S(f_5,f_4)[1,x_1,x_2,1],
&r_9=S(f_5,f_4)[1,x_2x_1^2,x_2^2x_1,1].&
\end{array}
\end{eqnarray*}}
We use a simple Program in Maple to help us calculate the remainder of these compositions modulo $G$ in $k\langle x_1, x_2\rangle$, and denote the results by $r_1, r_2, \cdots, r_9$, respectively.

Apply Lemma \ref{truncated diamond lemma}, as a necessary condition for $G$ to be Gr\"{o}bner, all the coefficients of $r_i$'s should be zero.  We use Maple to help solve the possible solutions to the system of equations given by setting the coefficients of $r_i$'s to $0$, together with $a_3 c_3\neq0,\ a_1=b_1=c_1$. By Lemma \ref{twisting}, it is equivalent to break the system into two situations:
\begin{enumerate}
\item[Case 1:] $b_1=0$. Maple Solving Command gives no solutions.
\item[Case 2:] $b_1=1$. Maple Solving Command gives one family of solutions. Twist the $b_1$ back and include them as a parameter $p$, we get the family of solutions $\mathcal{G}$ below.
\end{enumerate}

\begin{example}
Let $\mathcal{G}=\{\mathcal{G}(p, j): p\neq0,\ j^4=1\}$, where $\mathcal{G}(p, j)=k\langle x_1, x_2\rangle/(f_1, f_2, f_3)$
is the $\Z^2$-graded algebra with relations
{\small\begin{eqnarray*}
\begin{split}
&f_1=x_2x_1^3+px_1x_2x_1^2+p^2x_1^2x_2x_1+p^3x_1^3x_2,\\
&f_2=x_2^2x_1^2+px_2x_1x_2x_1+p^2x_2x_1^2x_2+p^2x_1x_2^2x_1
+p^3x_1x_2x_1x_2+p^4x_1^2x_2^2,\\
&f_3=x_2^3x_1+px_2^2x_1x_2+p^2x_2x_1x_2^2+p^3x_1x_2^3,\\
&f_4=x_2x_1x_2x_1^2+px_2x_1^2x_2x_1+p^2jx_1x_2x_1x_2x_1+p^3(j-1)x_1x_2x_1^2x_2\\
   &\hskip10mm +p^3(j-j^2)x_1^2x_2^2x_1+p^4(j-1)x_1^2x_2x_1x_2+p^5(-1+j-j^3)x_1^3x_2^2,\\
&f_5=x_2^2x_1x_2x_1+px_2x_1x_2^2x_1+p^2jx_2x_1x_2x_1x_2+p^3(j-j^2)x_2x_1^2x_2^2\\
   &\hskip10mm +p^3(j-1)x_1x_2^2x_1x_2+p^4(j-1)x_1x_2x_1x_2^2+p^5(-1+j-j^3)x_1^2x_2^3.
\end{split}
\end{eqnarray*}}
\end{example}

\begin{enumerate}
\item $\mathcal{G}(p, j)\in \mathcal{G}$ is isomorphic to a twisting $\Z^2$-graded algebra of $\mathcal{G}(1, j)$.
\item $\mathcal{G}(p, j)\in\mathcal{G}$ has no homogeneous normal element of total degree $\leq3$.
\item
If $\mathcal{G}(p, j)\in\mathcal{G}$, then $z_1:=x_1^4,\ z_2:=x_2^4,\ z_3:= (x_2x_1+px_1x_2)^2+2p^3(1-j^2)x_1^2x_2^2,\
z_4:=x_2x_1^2x_2x_1^2+p^4x_1^2x_2x_1^2x_2,\ z_5:=x_2^2x_1x_2^2x_1+p^4x_1x_2^2x_1x_2^2$ form a
normal sequence in $\mathcal{G}(p, j)$ such that the factor algebra
$\mathcal{G}(p, j)/(z_1, z_2, z_3, z_4, z_5)$ is finite dimensional.
\item If $\mathcal{G}(p, j)\in\mathcal{G}$ with $j^2=-1$, then elements
$w_4:=x_2x_1^2+p^2x_1^2x_2,\ w_5:=x_2^2x_1+p^2x_1x_2^2$ are
normal in $\mathcal{G}(p, j)/(z_1, z_2, z_3)$ such that the factor algebra
$\mathcal{G}(p, j)/(z_1, z_2, z_3, w_4, w_5)$ is also finite dimensional.
\item Algebras $\mathcal{G}(p, j)\in\mathcal{G}$ are AS-regular of global dimension $5$. They are all strongly noetherian, Auslander regular and Cohen-Macaulay.
\end{enumerate}

It is interesting to see that a minimal set of define relations of algebras in $\mathcal{G}$ is a Gr\"{o}bner basis, which is rare for AS-regular algebras of high global dimensions.

\section{Classification of type (4, 4, 4)}

It remains to consider 5-dimensional AS-regular $\Z^2$-graded algebras that are determined by three defining relations of the (total) degrees 4. The classification of this case has been worked out, by using $A_\infty$-algebraic method, under a generic condition in \cite{WW} already. For uniformity, we turn out the same result again as a reconfirmation in different ways. The AS-regular algebras of this type are listed from $\mathcal{H}$ to $\mathcal{P}$.

The minimal free resolution of the trivial module $k$ over $A^{\rm gr}$ is of the form
\begin{eqnarray*}
0
\to A^{\rm gr}(-10)
\to A^{\rm gr}(-9)^2
\to A^{\rm gr}(-6)^3
\to A^{\rm gr}(-4)^3
\to A^{\rm gr}(-1)^2
\to A^{\rm gr}
\to k_A
\to 0.
\end{eqnarray*}
So
$$
H_{(R^0)^{\rm gr}}(t)-H_{R^{\rm gr}}(t)=3t^4+\sum_{n\geq5}r_nt^n
$$
and then $G$ has exactly three relations of total degree $4$,
which are $f_1, f_2, f_3$. Hence
$G_{min}=\{f_1, f_2, f_3\}$
and the degrees of $f_1, f_2, f_3$ are all contained in
$\{(3, 1),(2, 2),(1, 3)\}$. It is impossible that any two of relations in $G$
are of the same degree $(3, 1)$ or $(1, 3)$, so
$$
\deg(f_2)=(2, 2)\quad \text{with}\quad \lw(f_2)\in \{x_2^2x_1^2,\ x_2x_1x_2x_1,\ x_2x_1^2x_2\}.
$$

\begin{lemma}
\label{lemma1-type-(4,4,4)}
It is impossible that $\deg(f_1)=\deg(f_3)=(2, 2)$.
\end{lemma}

\begin{proof}
If the degrees of $f_1, f_2, f_3$ are all of degree $(2, 2)$, then
by Lemma \ref{minimal resolution} the only possible
minimal free resolution of $k_A$ is of the form:
\begin{equation*}
0 \to
\begin{array}[t]{c}
A\\[-0.5em] \text{\tiny\rm (5, 5)}
\end{array}
\to
\begin{array}[t]{c}
A^2\\[-0.5em] \text{\tiny \rm (5, 4)}\\ [-0.7em] \text{\tiny \rm (4, 5)}
\end{array}
\to
\begin{array}[t]{c}
A^3\\[-0.5em] \text{\tiny \rm (3, 3)}\\ [-0.7em] \text{\tiny \rm (3, 3)}\\[-0.7em] \text{\tiny \rm (3, 3)}
\end{array}
\to
\begin{array}[t]{c}
A^3\\[-0.5em] \text{\tiny \rm (2, 2)}\\ [-0.7em] \text{\tiny \rm  (2, 2)}\\[-0.7em] \text{\tiny \rm (2, 2)}
\end{array}
\to
\begin{array}[t]{c}
A^2\\[-0.5em]\text{\tiny \rm (1, 0)}\\ [-0.7em] \text{\tiny \rm (0, 1)}
\end{array}
\to
\begin{array}[t]{c}
A\\[-0.5em] \text{\tiny \rm (0 ,0)}
\end{array}
\to k_A
\to 0.
\end{equation*}
Observe that $\lw(f_1)=x_2x_1^2x_2$, $\lw(f_2)=x_2x_1x_2x_1$, $\lw(f_3)=x_2^2x_1^2$, one gets
$$
H_{R^3}(t_1, t_2)-H_{R}(t_1, t_2)=t_1^2t_2^3
+\sum_{|\alpha|\geq6}r_{\alpha}t_1^{\alpha_1}t_2^{\alpha_2}.
$$
So $G$ has a unique element of total degree $5$, which is $f_4$ of degree $(2, 3)$. By checking all normal words of degree $(2, 3)$ modulo $G^3$,
one gets $\lw(f_4)=x_2x_1x_2^2 x_1$ and then
$$
H_{R^4}(t_1, t_2)-H_{R}(t_1, t_2)=3t_1^4t_2^3+2t_1^3t_2^4
+\sum_{|\alpha|\geq8}r_{\alpha}t_1^{\alpha_1}t_2^{\alpha_2}.
$$
So $G$ has exact three relations, say $f_5, f_6, f_7$,
of degree $(4, 3)$, and two relations $f_8, f_9$ of degree $(3, 4)$. However, the assumption that $A$
is a domain together with the observation
$$
\nw(G^4)_{(3,4)}=\{x_2x_1^3x_2^3, x_1x_2x_1x_2^3x_1, x_1^2x_2^4x_1,
x_1^2x_2^3x_1x_2, x_1^2x_2^2x_1x_2^2, x_1^2x_2x_1x_2^3, x_1^3x_2^4\}
$$
give rise to $\lw(f_8)=\lw(f_9)=x_2x_1^3x_2^3$, which is impossible.
\end{proof}

\begin{lemma}
$\deg(f_1)=(3, 1)$ with $\lw(f_1)=x_2x_1^3$, and $\deg(f_3)=(1, 3)$ with $\lw(f_3)=x_2^3x_1$.
\end{lemma}

\begin{proof}
If $\deg(f_3)=(2, 2)$, then $\deg(f_1)=(3, 1)$ with $\lw(f_1)=x_2x_1^3$.
By Lemma \ref{minimal resolution}
the only possible minimal free resolution of $k_A$ is of the form:
\begin{equation*}
0 \to
\begin{array}[t]{c}
A\\[-0.5em] \text{\tiny\rm (6, 4)}
\end{array}
\to
\begin{array}[t]{c}
A^2\\[-0.5em] \text{\tiny \rm (6, 3)}\\ [-0.7em] \text{\tiny \rm  (5, 4)}
\end{array}
\to
\begin{array}[t]{c}
A^3\\[-0.5em] \text{\tiny \rm (4, 2)}\\ [-0.7em] \text{\tiny \rm  (4, 2)}\\[-0.7em] \text{\tiny \rm (3, 3)}
\end{array}
\to
\begin{array}[t]{c}
A^3\\[-0.5em] \text{\tiny \rm (3, 1)}\\ [-0.7em] \text{\tiny \rm  (2, 2)}\\[-0.7em] \text{\tiny \rm (2, 2)}
\end{array}
\to
\begin{array}[t]{c}
A^2\\[-0.5em] \text{\tiny \rm (1, 0)}\\ [-0.7em] \text{\tiny \rm  (0, 1)}
\end{array}
\to
\begin{array}[t]{c}
A\\[-0.5em] \text{\tiny \rm (0, 0)}
\end{array}
\to k_A
\to 0.
\end{equation*}
The result follows immediately from the same discussion as that
in the proof of Lemma \ref{lemma2-type-(4,4,4,5)}.
\end{proof}

\begin{proposition}
\label{resolution-type-(4,4,4)}
The minimal free resolution of $k_A$ is of the form:
\begin{equation*}
0 \to
\begin{array}[t]{c}
A\\[-0.5em] \text{\tiny\rm (5, 5)}
\end{array}
\to
\begin{array}[t]{c}
A^2\\[-0.5em] \text{\tiny \rm (5, 4)}\\[-0.7em] \text{\tiny \rm (4, 5)}
\end{array}
\to
\begin{array}[t]{c}
A^3\\[-0.5em] \text{\tiny \rm (4, 2)}\\[-0.7em] \text{\tiny \rm (3, 3)}\\[-0.7em] \text{\tiny \rm (2, 4)}
\end{array}
\to
\begin{array}[t]{c}
A^3\\[-0.5em] \text{\tiny \rm (3, 1)}\\[-0.7em] \text{\tiny \rm (2, 2)}\\[-0.7em] \text{\tiny \rm (1, 3)}
\end{array}
\to
\begin{array}[t]{c}
A^2\\[-0.5em] \text{\tiny \rm (1, 0)}\\[-0.7em] \text{\tiny \rm (0, 1)}
\end{array}
\to
\begin{array}[t]{c}
A\\[-0.5em] \text{\tiny \rm (0, 0)}
\end{array}
\to k_A
\to 0.
\end{equation*}
\end{proposition}

\begin{proof}
It is easy to prove by applying Lemma \ref{minimal resolution}.
\end{proof}

\begin{lemma}
\label{lemma3-type-(4,4,4)}
$\lw(f_2)=x_2^2x_1^2$.
\end{lemma}

\begin{proof}
It is impossible that $\lw(f_2)=x_2x_1^2x_2$, since otherwise one will get a contradiction
$$
H_{R^3}(t_1, t_2)-H_{R}(t_1, t_2)=-t_1^4t_2^2
+\sum_{|\alpha|\geq7}r_{\alpha}t_1^{\alpha_1}t_2^{\alpha_2}.
$$
So it remains to eliminate the possibility of $\lw(f_2)=x_2x_1x_2x_1$.
If it is not the case, then
$$
H_{R^3}(t_1, t_2)-H_{R}(t_1, t_2)=t_1^3t_2^3
+\sum_{|\alpha|\geq7}r_{\alpha}t_1^{\alpha_1}t_2^{\alpha_2}
$$
and hence $G$ has exactly one element of total degree $6$,
which is $f_4$ of degree $(3, 3)$. By checking normal words of corresponding degrees modulo $G^3$, one can assume
\begin{eqnarray*}
\begin{split}
&f_1=x_2x_1^3+a_1x_1x_2x_1^2+a_2x_1^2x_2x_1+a_3x_1^3x_2,\\
&f_2=x_2x_1x_2x_1+b_2x_2x_1^2x_2+b_3x_1x_2^2x_1+b_4x_1x_2x_1x_2+b_5x_1^2x_2^2,\\
&f_3=x_2^3x_1+c_1x_2^2x_1x_2+c_2x_2x_1x_2^2+c_3x_1x_2^3.\\
\end{split}
\end{eqnarray*}

Let $r_1, r_2, r_3, r_4$ be the remainders of the following four compositions
$$
S(f_2, f_1)[1, x_1^2, x_2x_1, 1],\;  S(f_2, f_2)[1, x_2x_1, x_2x_1, 1],\; S(f_3, f_1)[1, x_1^2, x_2^2, 1],\; S(f_3, f_2)[1, x_2x_1, x_2^2, 1]
$$
modulo $G^3$ in $k\langle x_1, x_2\rangle$, respectively. Coefficients of $r_1$ and $r_4$
should be $0$ by Lemma \ref{truncated diamond lemma},
and these coefficient equations together with $a_3\neq0, c_3\neq0$ give
three families of solutions:
\begin{enumerate}
\item
$a_1=0, a_2=0, a_3=-p^3, b_2=0, b_3= 0, b_4=-p^2, b_5=0, c_1=0, c_2=0, c_3=p^3$.
Then
$$
r_2=-p^2x_2x_1^2x_2x_1x_2+p^2x_1x_2x_1x_2^2x_1,\quad  r_3=2p^9x_1^3x_2^3.
$$
\item
$a_1=0, a_2=0, a_3=p^3, b_2=0, b_3= 0, b_4=-p^2, b_5=0, c_1=0, c_2=0, c_3=p^3$.
Then
$$
r_2=-p^2x_2x_1^2x_2x_1x_2+p^2x_1x_2x_1x_2^2x_1,\quad r_3=0,
$$
and hence $f_4=x_2x_1^2x_2x_1x_2-x_1x_2x_1x_2^2x_1$ by Lemma \ref{truncated diamond lemma}.
Therefore
$$
S(f_4, f_2)[1, x_1, x_2x_1^2, 1]=-x_1(x_2x_1x_2^2x_1^2+p^5x_1^2x_2^2x_1x_2)
$$
\item
$a_1=e, a_2=p, a_3=ep, b_2=e, b_3=e, b_4=e^2, b_5=0, c_1=e, c_2=q, c_3=eq.$
Then
\begin{eqnarray*}
f_2&=&x_2x_1x_2x_1+ex_2x_1^2x_2+ex_1x_2^2x_1+e^2x_1x_2x_1x_2\\
   &=&(x_2x_1+ex_1x_2)^2.
\end{eqnarray*}
\end{enumerate}

All these three families of solutions contradict that $A$ is a domain and so
the result follows.
\end{proof}

\begin{remark}
In the proof of Lemma \ref{lemma3-type-(4,4,4)}, we use Maple program to help
us calculate $r_1, r_2, r_3, r_4$, and use Maple solving command to solve
the system of coefficient equations.
Observe that either $b_2=b_3$ or $b_4=b_2b_3$ (from the coefficient of $r_i$'s), we have justified that the less complicated system can be solved by hand to give the same set of solutions.
\end{remark}

\begin{proposition}
\label{type-(4,4,4)}
$G=\{f_1, f_2, f_3, f_4, f_5\}$ and $G_{min}=\{f_1, f_2, f_3\}$ with
{\small\begin{eqnarray*}
\begin{split}
&f_1=x_2x_1^3+a_1x_1x_2x_1^2+a_2x_1^2x_2x_1+a_3x_1^3x_2,\\
&f_2=x_2^2x_1^2+b_1x_2x_1x_2x_1+b_2x_2x_1^2x_2+b_3x_1x_2^2x_1+b_4x_1x_2x_1x_2 +b_5x_1^2x_2^2,\\
&f_3=x_2^3x_1+c_1x_2^2x_1x_2+c_2x_2x_1x_2^2+c_3x_1x_2^3,\\
&f_4=x_2x_1x_2x_1^2+d_1x_2x_1^2x_2x_1+d_2x_1x_2x_1x_2x_1+d_3x_1x_2x_1^2x_2 +d_4x_1^2x_2^2x_1 + d_5x_1^2x_2x_1x_2+d_6x_1^3x_2^2,\\
&f_5=x_2^2x_1x_2x_1+e_1x_2x_1x_2^2x_1+e_2x_2x_1x_2x_1x_2+e_3x_2x_1^2x_2^2 +e_4x_1x_2^2x_1x_2 + e_5x_1x_2x_1x_2^2+e_6x_1^2x_2^3,
\end{split}
\end{eqnarray*}}
where all $a_i, b_i, c_i, d_i, e_i\in k$. Moreover, one has $a_3 c_3\neq0$, $a_1\neq b_1$ and $c_1\neq b_1$.
\end{proposition}

\begin{proof}
By previous discussion we have $G_{min}=\{f_1, f_2, f_3\}$, and
by a similar discussion as that in the proof of Proposition \ref{type-(4,4,4,5,5)}, one can see that
$G=\{f_1, f_2, f_3, f_4, f_5\}$ with $f_1, f_2, \cdots, f_5$
are of the given form. Since $A$ is a domain we have $a_3c_3\neq0$.
Applying Lemma \ref{truncated diamond lemma}, the remainder of the composition
$S(f_2, f_1)[1, x_1, x_2, 1]$ (resp. $S(f_3, f_2)[1, x_1, x_2, 1]$) modulo $G^3$
is a non-zero scalar multiple of $f_4$ (resp. $f_5$) in $k\langle x_1, x_2\rangle$. It follows immediately $a_1\neq b_1$ and $c_1\neq b_1$, and then we finish the proof.
\end{proof}

Now we turn to find possible solutions of the coefficients of $f_i$ in Proposition \ref{type-(4,4,4)}. The relations of $G$ and their compositions
have exactly the same form of that for type (4, 4, 4, 5, 5). The difference is that the system of equations to be solved here consisting of these by setting the coefficients of $r_i$'s to $0$, together with $a_3 c_3\neq0,\ a_1\neq b_1$ and $b_1\neq c_1$. Moreover, we will always assume $b_3=0$ whenever $b_2b_3=0$, otherwise, we can pass to $A^{\omega}$. Apply Lemma \ref{twisting}, it is equivalent to  break the system into eight situations as follows:
\begin{enumerate}
\item[Case 1:] $b_1=1,\, b_2=b_3=0$. Maple Solving Command gives $2$ families of solutions. Twist the $b_1$ back and include them as a parameter $p$, we get the families of solutions $\mathcal{H}$ and $\mathcal{I}$ below.
\item[Case 2:] $b_1=1,\, b_2\neq0,\, b_3=0$. Maple Solving Command gives $1$ family of solutions. Twist the $b_1$ back and include them as a parameter $p$, we get the families of solutions $\mathcal{J}$ below.
\item[Case 3:] $b_1=1,\, b_2=b_3\neq0$. Maple Solving Command gives $2$ families of solutions. Twist the $b_1$ back and include them as a parameter $p$, we get the families of solutions $\mathcal{K},\,\mathcal{L}$ below.
\item[Case 4:] $b_1=1,\, b_2\neq b_3,\, b_2\neq 0,\, b_3\neq0$. Maple Solving Command gives no solution.
\item[Case 5:] $b_1=0,\, b_2=b_3=0,\, a_1=1$. Maple Solving Command gives $2$ families of solutions. Twist the $a_1$ back and include them as a parameter $p$, we get the families of solutions $\mathcal{M},\,\mathcal{N}$ below.
\item[Case 6:] $b_1=0,\, b_2\neq0,\, b_3=0,\, c_1=1$. Maple Solving Command gives $1$ family of solutions. Twist the $c_1$ back and include them as a parameter $p$, we get the family of solutions $\mathcal{O}$ below.
\item[Case 7:] $b_1=0,\, b_2=b_3\neq0,\, a_1=1$. Maple Solving Command gives $1$ family of solutions. Twist the $a_1$ back and include them as a parameter $p$, we get the family of solutions  $\mathcal{P}$ below.
\item[Case 8:] $b_1=0,\, b_2\neq b_3,\, b_2\neq0,\, b_3\neq0,\, a_1=1$. Maple Solving Command gives no solution.
\end{enumerate}

\begin{example}
Let $\mathcal{H}=\{\mathcal{H}(p, q): p\neq0,\ q\neq0\}$, where $\mathcal{H}(p, q)=k\langle x_1, x_2\rangle/(f_1, f_2, f_3)$
is the $\Z^2$-graded algebra with relations
\begin{eqnarray*}
\begin{split}
&f_1=x_2x_1^3-p^3q^3x_1^3x_2\\
&f_2=x_2^2x_1^2+px_2x_1x_2x_1-p^3q^2x_1x_2x_1x_2-p^4q^4x_1^2x_2^2\\
&f_3=x_2^3x_1-p^3q^3x_1x_2^3.
\end{split}
\end{eqnarray*}
\end{example}

\begin{enumerate}
\item
$\mathcal{H}(p, q)$ is isomorphic to a twisting $\Z^2$-graded algebra of $\mathcal{H}(1, q)$.
\item
$\mathcal{H}(p, q)\in\mathcal{H}$ corresponds to  $\mathbf{A}(\mathbf{l}_2, \mathbf{t})$ in \cite{WW}
by $\mathbf{l}_2=p,\ \mathbf{t}=-pq$.
\end{enumerate}

\begin{example}
Let $\mathcal{I}=\{\mathcal{I}(p, q): p\neq0,\ q\neq0\}$, where $\mathcal{I}(p, q)=k\langle x_1, x_2\rangle/(f_1, f_2, f_3)$
is the $\Z^2$-graded algebra with relations
\begin{eqnarray*}
\begin{split}
&f_1=x_2x_1^3+p(q+1)x_1x_2x_1^2+p^2(q^2+q)x_1^2x_2x_1+p^3q^3x_1^3x_2\\
&f_2=x_2^2x_1^2+px_2x_1x_2x_1-p^3q^2x_1x_2x_1x_2-p^4q^4x_1^2x_2^2\\
&f_3=x_2^3x_1+p(q+1)x_2^2x_1x_2+p^2(q^2+q)x_2x_1x_2^2+p^3q^3x_1x_2^3.
\end{split}
\end{eqnarray*}
\end{example}

\begin{enumerate}
\item $\mathcal{I}(p, q)$ is isomorphic to a twisting $\Z^2$-graded algebra of $\mathcal{I}(1, q)$.
\item $\mathcal{I}(p, -1)=\mathcal{H}(p, 1)$.
\item If $\mathcal{I}(p, q)\in\mathcal{I}$, then the element $z=x_2x_1^2+px_1x_2x_1+p^2q^2x_1^2x_2$
is regular normal in $\mathcal{I}(p, q)$ and the factor algebras
$\mathcal{I}(p, q)/(z)$ is isomorphic to $\mathbf{D}(\mathbf{v},\mathbf{p})$ in \cite{LPWZ} with $\mathbf{v}=p,\ \mathbf{p}=pq$.
\item $\mathcal{I}(p, q)\in\mathcal{I}$ corresponds to  $\mathbf{E}(\mathbf{p},\mathbf{t})$ in \cite{WW} by
$\mathbf{p}=p(q+1),\ \mathbf{t}=pq$.
\end{enumerate}

\begin{example}
Let $\mathcal{J}=\{\mathcal{J}(p, j): p\neq0,\ j^4+j^3+j^2+j+1=0\}$, where $\mathcal{J}(p, j)=k\langle x_1, x_2\rangle/(f_1, f_2, f_3)$
is the $\Z^2$-graded algebra with relations
{\small\begin{eqnarray*}
\begin{split}
&f_1=x_2x_1^3-pjx_1x_2x_1^2-p^2j(j^2+j+1)x_1^2x_2x_1+p^3(j^3-2j-2)x_1^3x_2\\
&f_2=x_2^2x_1^2+px_2x_1x_2x_1-p^2j(j^2+2j+2)x_2x_1^2x_2-p^3(j+1)^2x_1x_2x_1x_2
   -p^4(3j^2+5j+3)x_1^2x_2^2\\
&f_3=x_2^3x_1+p(1-j^2-j^3)x_2^2x_1x_2+p^2(1-2j^2-2j^3)x_2x_1x_2^2
   +p^3(1-2j^2-2j^3)x_1x_2^3.
\end{split}
\end{eqnarray*}}
\end{example}

\begin{enumerate}
\item $\mathcal{J}(p, j)$ is isomorphic to a twisting $\Z^2$-graded algebra of $\mathcal{J}(1, j)$.
\item If $\mathcal{J}(p, j)\in\mathcal{J}$, then the element $z=x_2x_1^2+px_1x_2x_1+p^2(j^2+j^3)^2x_1^2x_2$
is regular normal in $\mathcal{J}(p, j)$ and the factor algebras
$\mathcal{J}(p, j)/(z)$ is isomorphic to $\mathbf{D}(\mathbf{v}, \mathbf{p})$ in \cite{LPWZ} with
$\mathbf{v}=p,\ \mathbf{p}=-p(j^2+j^3)$.
\item $\mathcal{J}(p, j)\in\mathcal{J}$ corresponds to $\mathbf{I}(\mathbf{c}, \mathbf{g})$ in \cite{WW} by
$\mathbf{c}=-p(1+j),\ \mathbf{g}=j^3$.
\end{enumerate}

\begin{example}
Let $\mathcal{K}=\{\mathcal{K}(p, q): p\neq0,\ q\neq0,\ q\neq1\}$, where $\mathcal{K}(p, q)=k\langle x_1, x_2\rangle/(f_1, f_2, f_3)$
is the $\Z^2$-graded algebra with relations
\begin{eqnarray*}
\begin{split}
&f_1=x_2x_1^3+pqx_1x_2x_1^2+p^2q^2x_1^2x_2x_1+p^3q^3x_1^3x_2\\
&f_2=x_2^2x_1^2+px_2x_1x_2x_1+p^2qx_2x_1^2x_2+p^2qx_1x_2^2x_1
   +p^3q^2x_1x_2x_1x_2+p^4q^4x_1^2x_2^2\\
&f_3=x_2^3x_1+pqx_2^2x_1x_2+p^2q^2x_2x_1x_2^2+p^3q^3x_1x_2^3.
\end{split}
\end{eqnarray*}
\end{example}

\begin{enumerate}
\item $\mathcal{K}(p, q)$ is isomorphic to a twisting $\Z^2$-graded algebra of $\mathcal{K}(1, q)$.
\item $\mathcal{K}(p, q)\in\mathcal{K}$ corresponds to  $\mathbf{F}(\mathbf{l}_2, \mathbf{p})$ in \cite{WW} by
$\mathbf{l}_2=p,\ \mathbf{p}=pq$.
\end{enumerate}

\begin{example}
Let $\mathcal{L}=\{\mathcal{L}(p, q, r): p\neq0, \ q\neq0,\ r\neq0,\ q^3-(r+2)q^2+(2r^2+2r+1)q-r^2-r=0\}$,
where $\mathcal{L}(p, q, r)=k\langle x_1, x_2\rangle/(f_1, f_2, f_3)$
is the $\Z^2$-graded algebra with relations
\begin{eqnarray*}
\begin{split}
&f_1=x_2x_1^3+p(r+1)x_1x_2x_1^2+p^2q(r+1)x_1^2x_2x_1+p^3q^3x_1^3x_2\\
&f_2=x_2^2x_1^2+px_2x_1x_2x_1+p^2(q-r)(r+1)x_2x_1^2x_2+p^2(q-r)(r+1)x_1x_2^2x_1\\
   &\hskip10mm
   +p^3(q+r)(q-r^2-r)x_1x_2x_1x_2+p^4q^2(q-r^2-r)x_1^2x_2^2\\
&f_3=x_2^3x_1+p(r+1)x_2^2x_1x_2+p^2q(r+1)x_2x_1x_2^2+p^3q^3x_1x_2^3.
\end{split}
\end{eqnarray*}
\end{example}

\begin{enumerate}
\item $\mathcal{L}(p, q, r)$ is isomorphic to a twisting $\Z^2$-graded algebra of $\mathcal{L}(1, q, r)$.
\item $\mathcal{L}(p, q, -1)= \mathcal{H}(p, 1)$, here $q^2-q+1=0$.
\item $\mathcal{L}(-2, 1/2, 1/2)\in \mathcal{L}$ is the enveloping algebra of the $5$-dimensional $\Z^2$-graded Lie algebra $\mathfrak{g}={\rm Lie}(x_1, x_2)/(f_1, f_2, f_3)$, where ${\rm Lie}(x_1, x_2)$ is the $\Z^2$-graded free Lie algebra on $\{x_1, x_2\}$ with grading given by
$\deg(x_1)=(1, 0)$, $\deg(x_2)=(0, 1)$, and where $f_1=[[[x_2x_1]x_1]x_1]$,\, $f_2=[x_2[[x_2x_1]x_1]]$,\, $f_3=[x_2[x_2[x_2x_1]]]$.
\item $\mathcal{L}(p,q,r)\in\mathcal{L}$ with $r\neq-1$ and $q\neq 1/2$ corresponds to $\mathbf{G}(\mathbf{q}, \mathbf{r}, \mathbf{g})$ in \cite{WW} by
$\mathbf{q}=p^2q(r+1),\ \mathbf{r}=p^3q^3,\ \mathbf{g}=p^5q^4(qr-q^2-r^2+q-r)/r$. It is straightforward but tedious to check that $\mathbf{q}, \mathbf{r}, \mathbf{g}$ satisfy
the constraints given in \cite{WW} for algebras $\mathbf{G}$. Indeed, the inverse function is given by
{\small$$
p=\frac{\textbf{r}^2(\textbf{g}-\textbf{q}\textbf{r})}{\textbf{g}(\textbf{g} +\textbf{q}\textbf{r})},\;\; q=-\frac{\textbf{q}\textbf{r}\textbf{g}^2(\textbf{g}+\textbf{q}\textbf{r})}{ (\textbf{g}-\textbf{q}\textbf{r})(\textbf{r}^5 +\textbf{q}\textbf{r}\textbf{g}^2+\textbf{g}^3)},\;\;
r=-\frac{(\textbf{g}+\textbf{q}\textbf{r})(\textbf{r}^5 +\textbf{q}\textbf{r}\textbf{g}^2+\textbf{g}^3)}{\textbf{r}^5 (\textbf{g}-\textbf{q}\textbf{r})}-1.
$$}
A simple Maple code is given to help us do the calculation.
\end{enumerate}

\begin{example}
Let $\mathcal{M}=\{\mathcal{M}(p): p\neq0\}$, where $\mathcal{M}(p)=k\langle x_1, x_2\rangle/(f_1, f_2, f_3)$
is the $\Z^2$-graded algebra with relations
\begin{eqnarray*}
\begin{split}
&f_1=x_2x_1^3+px_1x_2x_1^2+p^2x_1^2x_2x_1+p^3x_1^3x_2\\
&f_2=x_2^2x_1^2-p^4x_1^2x_2^2\\
&f_3=x_2^3x_1+px_2^2x_1x_2+p^2x_2x_1x_2^2+p^3x_1x_2^3.
\end{split}
\end{eqnarray*}
\end{example}

\begin{enumerate}
\item $\mathcal{M}(p)$ is isomorphic to a twisting $\Z^2$-graded algebra of $\mathcal{M}(1)$.
\item If $\mathcal{M}(p)\in\mathcal{M}$, then the element $z=x_2x_1^2+p^2x_1^2x_2$
is regular normal in $\mathcal{M}(p)$ and the factor algebras
$\mathcal{M}(p)/(z)$ is isomorphic to $\mathbf{D}(\mathbf{v},\mathbf{p})$ in \cite{LPWZ} with $\mathbf{v}=0,\ \mathbf{p}=p$.
\item $\mathcal{M}(p)\in\mathcal{M}$ corresponds to $\mathbf{C}(\mathbf{p})$ in \cite{WW} by
$\mathbf{p}=p$.
\end{enumerate}

\begin{example}
Let $\mathcal{N}=\{\mathcal{N}(p): p\neq0\}$, where $\mathcal{N}(p)=k\langle x_1, x_2\rangle/(f_1, f_2, f_3)$
is the $\Z^2$-graded algebra with relations
\begin{eqnarray*}
\begin{split}
&f_1=x_2x_1^3+px_1x_2x_1^2+p^2x_1^2x_2x_1+p^3x_1^3x_2\\
&f_2=x_2^2x_1^2+p^4x_1^2x_2^2\\
&f_3=x_2^3x_1+px_2^2x_1x_2+p^2x_2x_1x_2^2+p^3x_1x_2^3.
\end{split}
\end{eqnarray*}
\end{example}

\begin{enumerate}
\item $\mathcal{N}(p)$ is isomorphic to a twisting $\Z^2$-graded algebra of $\mathcal{N}(1)$.
\item $\mathcal{N}(p)\in\mathcal{N}$ corresponds to $\mathbf{B}(\mathbf{p})$ in \cite{WW} by
$\mathbf{p}=p$.
\end{enumerate}

\begin{example}
Let $\mathcal{O}=\{\mathcal{O}(p, j): p\neq0,\ j^2+1=0\}$, where $\mathcal{O}(p, j)=k\langle x_1, x_2\rangle/(f_1, f_2, f_3)$
is the $\Z^2$-graded algebra with relations
\begin{eqnarray*}
\begin{split}
&f_1=x_2x_1^3+pjx_1x_2x_1^2-p^2jx_1^2x_2x_1+p^3x_1^3x_2\\
&f_2=x_2^2x_1^2+p^2(1-j)x_2x_1^2x_2-p^4jx_1^2x_2^2\\
&f_3=x_2^3x_1+px_2^2x_1x_2+p^2x_2x_1x_2^2+p^3x_1x_2^3.
\end{split}
\end{eqnarray*}
\end{example}

\begin{enumerate}
\item $\mathcal{O}(p, j)$ is isomorphic to a twisting $\Z^2$-graded algebra of $\mathcal{O}(1, j)$.
\item If $\mathcal{O}(p, j)\in\mathcal{O}$, then the element $z=x_2x_1^2-p^2jx_1^2x_2$
is regular normal in $\mathcal{O}(p, j)$ and the factor algebra
$\mathcal{O}(p, j)/(z)$ is isomorphic to $\mathbf{B}(\mathbf{p}, \mathbf{j})$ in \cite{LPWZ} with $\mathbf{p}=p,\ \mathbf{j}=-j$.
\item $\mathcal{O}(p, j)\in\mathcal{O}$ corresponds to $\mathbf{H}(\mathbf{i}, \mathbf{p})$ in \cite{WW} by
$\mathbf{p}=p,\ \mathbf{i}=j$.
\end{enumerate}

\begin{example}
Let $\mathcal{P}=\{\mathcal{P}(p, j): p\neq0,\ j^2-j+2=0\}$, where $\mathcal{P}(p, j)=k\langle x_1, x_2\rangle/(f_1, f_2, f_3)$
is the $\Z^2$-graded algebra with relations
\begin{eqnarray*}
\begin{split}
&f_1=x_2x_1^3+px_1x_2x_1^2+p^2jx_1^2x_2x_1-p^3(j+2)x_1^3x_2\\
&f_2=x_2^2x_1^2+p^2(j-1)x_2x_1^2x_2+p^2(j-1)x_1x_2^2x_1
   -p^3(j+1)x_1x_2x_1x_2-p^4(j-2)x_1^2x_2^2\\
&f_3=x_2^3x_1+px_2^2x_1x_2+p^2jx_2x_1x_2^2-p^3(j+2)x_1x_2^3.
\end{split}
\end{eqnarray*}
\end{example}

\begin{enumerate}
\item $\mathcal{P}(p, j)$ is isomorphic to a twisting $\Z^2$-graded algebra of $\mathcal{P}(1, j)$.
\item $\mathcal{P}(p, j)\in\mathcal{P}$ corresponds to $\mathbf{D}(\mathbf{p},\mathbf{q})$ in \cite{WW} by
$\mathbf{p}=p,\ \mathbf{q}=p^2j$.
\end{enumerate}

The correspondence given in the above discussions together with \cite[Theorem 5.2, Theorem 5.4, Theorem 5.5, Theorem 5.8, Theorem 5.9]{WW} gives rise to the following result.

\begin{proposition}
Algebras in the families from $\mathcal{H}$ to $\mathcal{P}$ are AS-regular of global dimension $5$. They are all strongly noetherian, Auslander regular and Cohen-Macaulay.
\end{proposition}

\section{Proof and Question}

Our main results follow immediately from the discussions in previous sections. Here we give a short summary of the discussions to justify the statements.

\begin{proof}
[Proof of Theorem \ref{Thm A}]
Firstly note that all algebras in $\mathfrak{X}$ are AS-regular properly $\Z^2$-graded algebras of global dimension five which are domains with two generators and of GK-dimension five by showing that each one of them is either a normal extension of four dimensional AS-regular algebras, or has enough normal elements, or is an iterated Ore extension of polynomial rings \cite[Subsection 5.2]{WW}. Secondly if $A, A'$ are two algebras in $\mathfrak{X}$ and $f:A\to A'$ is an isomorphism of $\Z^2$-graded algebras, then $f$ is determined by $f(x_i)=a_ix_i$ for $i=0,1$, where $a_i\in k$ is nonzero. It is easy to see that such $f$ exists iff $A=A'$, and so algebras in $\mathfrak{X}$ are pairwise non-isomorphic as $\Z^2$-graded algebras. Lastly, by previous discussions, if $A$ is an AS-regular properly $\Z^2$-graded algebra of global dimension five which is a domain with two generators and of GK-dimension at least four, then $A$ or $A^{\omega}$ falls into $\mathfrak{X}$ up to isomorphism of $\Z^2$-graded algebras. Now the result follows.
\end{proof}

\begin{proof}
[Proof of Theorem \ref{Thm B}]
We have already seen that all algebras in $\mathfrak{X}$ are strongly noetherian,
AS-regular and Cohen-Macaulay. Observe that these properties are preserved
under switching, the result follows readily from
Theorem \ref{Thm A}.
\end{proof}

\begin{proof}
[Proof of Corollary \ref{Thm C}]
Algebras in $\mathcal{G}$ fulfill the requirement of the first statement, while Proposition \ref{type-(4,4,4,5)} justifies the second one.
\end{proof}

Finally, we mention briefly the question raised in the introduction.
There are plenty of AS-regular algebras of which the obstructions
consisted of Lyndon words. For example, the quantum plane
$k\langle x_1, x_2\rangle/(x_2x_1-qx_1x_2)$ and the Jordan plane $k\langle x_1, x_2\rangle/(x_2x_1-x_1x_2-x_1^2)$; the algebras of type S1 and type S2 in \cite{ASc}; the algebras obtained in \cite{LPWZ,RZ,WW}; the binomial skew polynomial rings defined in \cite{GI2}; the universal enveloping algebra of finite dimensional one-generated graded Lie algebras.
The following result provides another evidence of the positive answer of Question \ref{question D}.

\begin{theorem}
Let $k\langle x_1, x_2\rangle$ be $\Z^2$-graded free algebra with $\deg(x_1)=(1, 0)$ and $\deg(x_2)=(0, 1)$. Let $\mathfrak{a}$ be a homogeneous ideal of $k\langle x_1,x_2\rangle$, $A=k\langle x_1,x_2\rangle/\mathfrak{a}$ and $G$ the reduced Gr\"{o}bner basis of $\mathfrak{a}$ with respect to the deg-lex order. Then $\lw(G)$ consists of Lyndon words if the algebra $A$ is an AS-regular domain either of global dimension $\leq 4$ or of global dimension $5$ and GK-dimension $\geq4$.
\end{theorem}
\begin{proof}
Let $A$ be an AS-regular domain of global dimension $\leq 4$. By a similar (but much simpler) discussion we have taken in Section 3 - Section 7, one can conclude the following:
\begin{enumerate}
\item $\lw(G)=\{x_2x_1\}$ if $A$ is an AS-regular domain of global dimension $2$;
\item $\lw(G)=\{x_2x_1^2,\, x_2^2x_1\}$ if $A$ is an AS-regular domain of global dimension $3$;
\item $\lw(G)=\{x_2x_1^2,\, x_2^3x_1,\, x_2^2x_1x_2x_1\}$ or $\{x_2^2x_1,\, x_2x_1^3,\, x_2x_1x_2x_1^2\}$ if $A$ is an AS-regular domain of global dimension $4$.
\end{enumerate}

Suppose that $A$ is an AS-regular domain of global dimension $5$ and $\gkdim A\geq4$.

If $\deg_1(f_1)\geq\deg_2(f_1)$, we see, from the discussions in previous sections, that $\lw(G)$ can be chosen as one of the following:
\begin{eqnarray*}
\begin{split}
&(4)\ \lw(G)=\{x_2x_1^2,\, x_2^2x_1x_2x_1,\, x_2^4x_1,\, x_2^3x_1x_2^2x_1\}, \; \mathrm{or}\\
&(5)\ \lw(G)=\{x_2x_1^2,\, x_2^3x_1,\, x_2^2x_1x_2x_1x_2x_1,\, x_2^2x_1x_2^2x_1x_2x_1\}, \; \mathrm{or}\hskip37mm\\
&(6)\ \lw(G)=\{x_2x_1^3,\, x_2^2x_1^2,\, x_2^3x_1,\, x_2x_1x_2x_1^2,\, x_2^2x_1x_2x_1\}.
\end{split}
\end{eqnarray*}

If $\deg_1(f_1)<\deg_2(f_1)$, we see, by a symmetric discussion in previous sections, that $\lw(G)$ can be chosen as one of the following:
\begin{eqnarray*}
\begin{split}
&(7)\ \lw(G)=\{x_2^2x_1, x_2x_1^4,\, x_2x_1x_2x_1^2,\, x_2x_1^2x_2x_1^3\}, \; \mathrm{or}\\
&(8)\ \lw(G)=\{x_2^2x_1,\, x_2x_1^3,\, x_2x_1x_2x_1x_2x_1^2,\, x_2x_1x_2x_1^2x_2x_1^2\}.\hskip42mm
\end{split}
\end{eqnarray*}

Observe all words have listed are Lyndon words, the result follows.
\end{proof}

\begin{remark} The method we used in this paper is also available in classifying the AS-regular $\Z^2$-graded algebra of global dimension $\leq4$ which is a domain with two generators of degrees $(1, 0)$ and $(0, 1)$, respectively.  We have justified it without using the Maple. In the case of global dimension $2$, our method gives one family of solutions, which turns out to be the quantum planes. In case of global dimension $3$, our method gives two families of solutions, which turns out to be the algebras of type S1 and type S2 in \cite{ASc}. In case of global dimension $4$, our method gives the same set of solutions of that given in \cite{LPWZ}.
\end{remark}

\vskip7mm

{\it Acknowledgments.}  This research is supported by the NSFC
(Grant No. 11271319) and partially by the NSF of Zhejiang Province of China (Grant No.LQ12A01028).

\end{document}